\documentclass[a4paper,reqno]{amsart}

\usepackage{enumerate}
\usepackage[colorlinks=true]{hyperref}
\hypersetup{urlcolor=blue, citecolor=cyan}
\hypersetup{citecolor=blue}

\numberwithin{equation}{section}

\newtheorem{thm}{Theorem}[section]

\newtheorem{prop}[thm]{Proposition}
\theoremstyle{definition}
\newtheorem{rem}[thm]{Remark}

\newcommand\R{{\mathbb R}}
\newcommand\C{{\mathbb C}}
\newcommand\N{{\mathbb N}}

\newcommand\Tma{T_{\mathrm{max}}}

\newcommand\Ens{{\mathcal E}}
\newcommand\Spa{{\mathcal X}}

\newcommand\dist{{\mathrm{d}}}

\newcommand\Srn{{\mathcal S}(\R^N )}

\newcommand\Imqu{{m}}
\newcommand\Imqd{{n}}
\newcommand\Imqt{{k}}

\newcommand\Imqc{{j}}
\newcommand\Imqs{{\mu}}
\newcommand\Imqp{{\nu}}
\newcommand\Imqh{{J}}

\newcommand\CTd{C_1}
\newcommand\CTc{C_2}
\newcommand\CTu{C_3}
\newcommand\CTs{C_4}
\newcommand\CTq{C_5}

\newcommand\goto{\mathop{\longrightarrow}}

\newcommand\MScN[1]{\href{http://www.ams.org/mathscinet-getitem?mr=#1}{\nolinkurl{(#1)}}}
\newcommand\DOI[1]{\href{http://dx.doi.org/#1}{(doi: \nolinkurl{#1})}}
\newcommand\LINK[1]{\href{#1}{(link: \nolinkurl{#1})}}

\newcommand\DI{u_0 }
\newcommand\DIb{v_0 }

\begin{document}

\title{Modified scattering for the critical nonlinear Schr\"o\-din\-ger equation}

\def\shorttitle{Modified scattering}

\author[T. Cazenave]{Thierry Cazenave$^1$}
\email{\href{mailto:thierry.cazenave@upmc.fr}{thierry.cazenave@upmc.fr}}

\author[I. Naumkin]{Ivan Naumkin$^2$}
\email{\href{mailto:ivan.naumkin@unice.fr}{ivan.naumkin@unice.fr}}

\address{$^1$Universit\'e Pierre et Marie Curie \& CNRS, Laboratoire Jacques-Louis Lions,
B.C. 187, 4 place Jussieu, 75252 Paris Cedex 05, France}

\address{$^2$Laboratoire J.A. Dieudonn\'e,  UMR CNRS 7351, Universit\'e de Nice Sophia-Antipolis, Parc Valrose,  06108 Nice Cedex 02, France}



\thanks{Ivan Naumkin thanks the project ERC-2014-CdG 646.650 SingWave for its financial support, and the Laboratoire J.A. Dieudonn\'e of the Universit\'e de Nice Sophia-Antipolis for its kind hospitality.}

\begin{abstract}

We consider the nonlinear Schr\"o\-din\-ger equation 
\begin{equation*} 
iu_t + \Delta u= \lambda  |u|^{\frac {2} {N}} u  
\end{equation*} 
in all dimensions $N\ge 1$, where $\lambda \in \C$ and $\Im \lambda \le 0$. We construct a class of initial values for which the corresponding solution is global and decays as $t\to \infty $, like $t^{- \frac {N} {2}}$ if $\Im \lambda =0$ and like $(t \log t)^{- \frac {N} {2}}$ if $\Im \lambda <0$. Moreover, we give an asymptotic expansion of those solutions as $t\to \infty $.
We construct solutions that do not vanish, so as to avoid any issue related to the lack of regularity of the nonlinearity at $u=0$. To study the asymptotic behavior, we apply the pseudo-conformal transformation and estimate the solutions by allowing a certain growth of the Sobolev norms which depends on the order of regularity through a cascade of exponents.

\end{abstract}

\maketitle

%\tableofcontents

\section{Introduction}
In this article, we consider the nonlinear Schr\"o\-din\-ger  equation
\begin{equation} \label{NLS1}
\begin{cases}
iu_{t} + \Delta u = \lambda|u|^{\alpha}u \\
u(0,x)= \DI 
\end{cases}
\end{equation}
on $\R^N $, where 
\begin{equation} \label{fCA1} 
\alpha=\frac {2} {N}
\end{equation} 
 and 
\begin{equation}  \label{fCL1} 
  \Im \lambda \le 0
\end{equation} 
and its equivalent integral formulation
\begin{equation} \label{NLSI} 
u(t)= e^{it \Delta } \DI - i \lambda  \int _0^t e^{i (t-s) \Delta }  |u|^\alpha u \, ds
\end{equation} 
where $(e^{it \Delta }) _{ t\in \R }$ is the Schr\"o\-din\-ger group. 

It is well known that the Cauchy problem for~\eqref{NLS1}--\eqref{fCL1}  is globally well posed in a variety of spaces, for instance in $H^1 (\R^N ) $, in $L^2 (\R^N ) $, and in
\begin{equation} \label{sigma}
\Sigma = H^1 (\R^N ) \cap L^2 (\R^N  , |x|^2 dx).
\end{equation}
See e.g.~\cite{Kato}. 
Concerning the long time asymptotic behavior of the solutions, $\alpha =\frac {2} {N}$ is a limiting case. 
Indeed, for $\alpha >\frac {2} {N}$, there is low energy scattering, i.e. a solution of~\eqref{NLS1}  with a sufficiently small initial value (in some appropriate sense) is asymptotic as $t\to \infty $ to a solution of the free Schr\"o\-din\-ger equation. See~\cite{Strauss2, GinibreV1, GinibreV2, CazenaveW1, GinibreOV, NakanishiO, CN}.
On the other hand, if $\alpha \le \frac {2} {N}$, then low energy scattering cannot be expected, see~\cite[Theorem~3.2 and Example~3.3, p.~68]{Strauss} and~\cite{Barab}. 

In the case $\alpha =\frac {2} {N}$, the relevant notion is modified scattering, i.e. standard scattering modulated by a phase. 
When $\Im \lambda =0$, the existence of modified wave operators was established in~\cite{Ozawa1} in dimension $N=1$. More precisely, for all sufficiently small asymptotic state $u^+$, there exists a solution of~\eqref{NLS1} which behaves as $t\to \infty $ like $ e^{i \phi (t, \cdot )} e^{t\Delta }u^+$, where the phase $\phi $ is given explicitly in terms of $u^+$. (See also~\cite{Carles}. See~ \cite{HNST, ShimomuraT} for extensions in dimension $N=2$.)
Conversely, for small initial values, it was proved in~\cite{HayashiNau1} that the asymptotic behavior of the corresponding solution has this form when $\Im \lambda =0$, in dimensions $N=1, 2, 3$. (See also~\cite{KitaW}.) 
If $\Im \lambda <0$, then the nonlinearity has some dissipative effect, and an extra log decay appears in the description of the asymptotic behavior of the solutions. This was established in space dimensions $N=1, 2, 3$ in~\cite{Shimomura}. (See also~\cite{HayashiNau2, HayashiNau3} for related results.)

Our purpose in this article is to complete the previous results for~\eqref{NLS1}-\eqref{fCA1}. In order to state our results, we introduce some notation. 
We consider three integers ${k},{m},{n}$ such that 
\begin{equation} \label{n2} 
\Imqt > \frac {N} {2}, \quad \Imqd > \max \Bigl\{ \frac {N} {2} +1,  \frac {N} {2\alpha  } \Bigr\}, \quad  2 \Imqu \ge \Imqt + \Imqd +1
\end{equation} 
and we let
\begin{equation} \label{n11}
\Imqh = 2\Imqu +2 + \Imqt+ \Imqd .
\end{equation}
We consider  the Banach space $\Spa $ introduced in~\cite[formulas~(1.6) and~(1.7)]{CN}, i.e.
\begin{equation} \label{n27} 
\begin{split} 
\Spa=  \{ u\in H^\Imqh  (\R^N ); & \, \langle x\rangle ^\Imqd D^\beta u \in L^\infty  (\R^N )  \text{ for  }  0\le  |\beta |\le 2\Imqu  \\   \langle x\rangle ^\Imqd D^\beta u  \in L^2 & (\R^N )  \text{ for  }   2\Imqu +1 \le  |\beta | \le 2\Imqu +2 + \Imqt,  \\  \langle x\rangle ^{\Imqh -  |\beta |} D^\beta u & \in L^2  (\R^N )  \text{ for  }    2\Imqu +2 + \Imqt <  |\beta | \le J \}
\end{split} 
\end{equation} 
with
\begin{equation}  \label{n28}
 \| u \|_\Spa =  \sum_{ \Imqc =0 }^{2\Imqu  }   \sup _{  |\beta |=    \Imqc }  \| \langle \cdot \rangle ^\Imqd D^\beta  u  \| _{ L^\infty  } +   \sum_{\Imqp =0 }^{\Imqt +1} \sum_{ \Imqs =0 }^\Imqd \sup_{  |\beta  |=\Imqp + \Imqs  +  2 \Imqu  +1  }   \| \langle \cdot  \rangle ^{\Imqd -\Imqs } D^\beta  u \| _{ L^2 } 
\end{equation}
where
\begin{equation*} 
\langle x\rangle=(1+|x|^2 )^{\frac{1}{2}}.
\end{equation*} 
Our main results are the following.

\begin{thm} \label{T1}
Let $\lambda \in \R$.
Assume~\eqref{fCA1},   \eqref{n2}, \eqref{n11}, let $\Spa $ be defined
by~$\eqref{n27}$-$\eqref{n28}$, and ${\Sigma}$ by $\eqref{sigma}$. Suppose that
$ \DI (x) =e^{i\frac{b|x|^2 }{4}}\DIb (x)$, where $b\in \R $ and $\DIb 
\in \Spa $ satisfies
\begin{equation} \label{fTPA11} 
\inf_{x\in\R^N }\langle x\rangle^{n}\left|  \DIb (x)\right|  >0.
\end{equation} 
If $b>0$ is sufficiently large, then there exists a unique, global solution $u$ in the class
$C([0,\infty),{\Sigma}) \cap L^\infty ((0,\infty)\times\R^N ) \cap L^\infty ((0,\infty), H^1 (\R^N ) )$ of~\eqref{NLSI}. Moreover,
 there exist $\delta >0$ and $w_0  \in L^\infty  (\R^N )$ with  $\langle \cdot \rangle ^n w_0  \in L^\infty  (\R^N )$ and $h\not \equiv 0$ such that
\begin{equation} \label{T1n2:2}
\|  u (t, \cdot ) - z(t, \cdot ) \|  _{L^2} +  (1+t)^{\frac {N} {2} } \|  u (t, \cdot ) - z (t, \cdot )  \|  _{L^{\infty}} \le C (1+t)^{-\delta }
\end{equation}
where 
\begin{equation*} 
z(t,x)= (1+bt)^{-\frac{N}{2}} e^{i\Phi (  t,  \cdot  )  } w_0 \Bigl(
\frac{ \cdot }{1+bt} \Bigr)
\end{equation*} 
and
\begin{equation*} 
\Phi (  t,x )  = \frac {b|x|^2 } {4(1+bt)} - \frac {\lambda} {b} \Bigl| 
w_0 \Bigl(  \frac{x}{1+bt} \Bigr)  \Bigr|  ^{\frac {2} {N}} \log ( 1+bt ) .
\end{equation*} 
In addition, 
\begin{equation} \label{T1n1}
 t^{ \frac{N}{2}} \| u (t) \| _{ L^\infty  } \goto _{ t\to \infty  }  b^{- \frac {N} {2}}  \| w_0\| _{ L^\infty  }  .
\end{equation}
\end{thm}

\begin{thm} \label{T2}
Let $\lambda \in \C$ with $\Im \lambda <0$.
Assume~\eqref{fCA1},   \eqref{n2}, \eqref{n11}, let $\Spa $ be defined
by~$\eqref{n27}$-$\eqref{n28}$, and ${\Sigma}$ by $\eqref{sigma}$. Suppose
$ \DI (x) =e^{i\frac{b|x|^2 }{4}}\DIb (x)$, where $b\in \R $ and $\DIb 
\in \Spa $ satisfies~\eqref{fTPA11}.
If  $b>0$ is sufficiently large, then there exists a unique, global solution $u\in
C([0,\infty), \Sigma ) \cap L^\infty ((0,\infty)\times\R^N ) \cap L^\infty ((0,\infty), H^1 (\R^N )  )$ of~\eqref{NLSI}. Moreover,
there exist $\delta >0$ and $f_{0}, w_0 \in L^{\infty},$ with $f_0$ real valued,  $ \|
f_{0}  \| _{L^{\infty}} \le \frac{1}{2}$, $w_0 \not \equiv 0$ and $ \langle \cdot \rangle ^n w_0\in L^\infty  (\R^N )  $ such that
\begin{equation}  \label{T1n4:2}
\|  u (t, \cdot ) - z(t, \cdot ) \|  _{L^2} +  (1+t)^{\frac {N} {2} } \|  u (t, \cdot ) - z (t, \cdot )  \|  _{L^{\infty}} \le C (1+t)^{-\delta }
\end{equation} 
where 
\begin{equation*} 
z (t, x)= (1+bt) ^{-\frac{N}{2} 
}e^{i \Theta (  t, \cdot )  }  \Psi \Bigl( t, \frac{ \cdot }{1+bt}\Bigr)  w_0 \Bigl(  \frac{ \cdot }{1+bt}\Bigr) 
\end{equation*} 
with
\begin{equation*} 
\Theta (  t,x ) =\frac{ b|x|^2 }{4(1+bt)} -  \frac
{\Re \lambda}{ \Im\lambda} \log  \Bigl(  \Psi \Bigl(  t, \frac {x} {1+bt}  \Bigr) \Bigr)
\end{equation*} 
and 
\begin{equation*} 
\Psi (  t, y )   =   \Bigl(   \frac{ 1 + f_0 (y ) }{1+ f_0 ( y )+ \frac {2 |\Im \lambda |} {Nb}  |
\DIb ( y )  | ^{\frac {2} {N}}  \log(1+bt)   }  \Bigr)^{\frac {N} {2 }} .
\end{equation*} 
In addition, 
\begin{equation} \label{T1n3}
 ( t \log t )^{ \frac {N} {2 }}  \| u ( t ) \| _{ L^\infty  } \goto  _{ t\to \infty  } ( \alpha  | \Im \lambda |)^{- \frac {N} {2}} .
\end{equation}
\end{thm}

\begin{rem} \label{eRM1} 
Here are some comments on the above Theorems~\ref{T1} and~\ref{T2}.
\begin{enumerate}[{\rm (i)}] 

\item \label{eRM1:1} 
The results are valid in any space dimension $N\ge 1$.

\item \label{eRM1:2} 
We do not require the initial value $\DI $ to have small amplitude. Instead, we require $\DI $ to be sufficiently oscillatory (in the sense that $b$ is requested to be sufficiently large).
Note also that Theorems~\ref{T1} and~\ref{T2} do not yield any information on the behavior of the solution for $t<0$.

\item \label{eRM1:3} 
It is easy to verify that  $\Srn \subset \Spa$, and that if $\rho \ge n$, then $\langle x \rangle ^{-\rho } \in \Spa$. Therefore, if $\DIb = c  \langle \cdot \rangle ^{-n} + \varphi $ with $c\in \C$, $c\not = 0$, and $\varphi \in \Srn$, $ |\varphi | \le ( |c| -\varepsilon )  \langle \cdot \rangle ^{-n} $, $\varepsilon >0$, then $\DIb \in  \Spa$ and $\DIb$ satisfies~\eqref{fTPA11}. 

\item \label{eRM1:3:1} 
The exponent $\delta >0$ that we obtain in~\eqref{T1n2:2} and~\eqref{T1n4:2} is provided by Proposition~\ref{p4} below and equals $1- \sigma _J$, where $\sigma _J$ is given by~\eqref{fCA9}. In particular, it is independent of the solution. Moreover, it can be chosen as close to $1$ as we want.

\item \label{eRM1:4} 
Note that the limit in~\eqref{T1n3} is independent of the initial value $\DI$. 
This is due to the fact that  the limit in~\eqref{fp4:2} is independent of the initial value $\DIb$ in~\eqref{NLS2}. This last property can be understood by considering the ODE $ i z'= \lambda  (1-bt)^{-1}  |z|^ \alpha z$. One easily verifies that $ |z(t)|^{-\alpha } =  | z(0) |^{-\alpha } + \frac {\alpha  |\Im \lambda |} {b}  |\log (1-bt)|$, so that $\lim  _{ t\uparrow \frac {1} {b} }  |\log (1-bt)|^{\frac {1} {\alpha }}  |z(t)| = ( \frac {b} {\alpha  |\Im \lambda |})^{\frac {1} {\alpha }}$ is independent of $z(0)\not = 0$. 

\item \label{eRM1:5} 
One can express formula~\eqref{T1n2:2} in the form of the standard modified scattering. To see this, let the dilation operator $D_a$ and the multiplier $M_a$ be defined by $D_a \phi (x)= a^{-\frac {N} {2}} \varphi (a^{-1 } x)$ and $M_a (x)= e^{i \frac { |x|^2} {4a}}$, so that (see~\cite{HayashiO}) $e^{it \Delta } = i^{- \frac {N} {2}} M_t D_t {\mathcal F} M_t$, where ${\mathcal F}$ is the Fourier transform. Using the 
relations $D_a {\mathcal F}= {\mathcal F} D _{ \frac {1} {a} }$ and $M_a D_b = D_b M _{ \frac {a} {b^2} }$, one obtains
\begin{equation*} 
e^{- i t \Delta } M _{ \frac {1+bt} {b} } D_{ 1+ bt } =   M _{ \frac {1} { b } } e^{- i \frac {t}  { 1+bt } \Delta }. 
\end{equation*} 
Since $z$ in~\eqref{T1n2:2} can be written in the form
\begin{equation*} 
z (t)= e^{- i  \frac {\lambda} {b} | 
w_0 (  \frac{x}{1+bt} )  |  ^{\frac {2} {N}} \log ( 1+bt )} e^{ i t \Delta } e^{- i t \Delta } M _{ \frac {1+bt} {b} } D_{ 1+b t } w_0
\end{equation*} 
we deduce that
\begin{equation*} 
e^{ - i t \Delta } [ e^{ i  \frac {\lambda} {b} | 
w_0 (  \frac{x}{1+bt} )  |  ^{\frac {2} {N}} \log ( 1+bt )}  z (t) ] \goto _{ t\to \infty  }   M _{ \frac {1} { b } } e^{- i \frac {1}  { b } \Delta } w_0 =: u^+
\end{equation*} 
in $L^2 (\R^N ) $. Therefore, \eqref{T1n2:2} takes the form of modified scattering. In other words, $u(t)$ behaves like $e^{ i  \frac {\lambda} {b} | 
w_0 (  \frac{x}{1+bt} )  |  ^{\frac {2} {N}} \log ( 1+bt )}  e^{t \Delta } u^+$, i.e. a free solution modulated by a phase.

\end{enumerate} 
\end{rem}

\begin{rem}  \label{eRM2} 
Here are some open questions related to Theorems~\ref{T1} and~\ref{T2}.
\begin{enumerate}[{\rm (i)}] 

\item  \label{eRM2:1}
We do not know what happens if $\Im \lambda >0$. 
Let us observe that if $\alpha <\frac {2} {N}$ and  $\Im \lambda >0$, then it follows from~\cite[Theorem~1.1]{CCDW} that every nontrivial solution of~\eqref{NLS1} either blows up in finite time or else is global with unbounded $H^1$ norm. The proof in~\cite{CCDW} apparently does not apply to the case $\alpha =\frac {2} {N}$. 
See also Remark~\ref{eRem4} below.

\item  \label{eRM2:2}
For equation~\eqref{NLS1} with $\Im \lambda >0$, it seems that no finite time blowup result is available (for any dimension $N$ and any $\alpha >0$). Note that for the same equation set on a bounded domain $\Omega $ with Dirichlet boundary conditions, there is no global solution for any $\alpha >0$. See~\cite[Section~2]{CCDW}.

\item  \label{eRM2:3}
If $\alpha < \frac {2} {N}$ and $\Im \lambda \le 0$, it seems that no precise description of the asymptotic behavior of the solutions of~\eqref{NLS1} is available. When $\lambda \in \R$, $\lambda >0$, it is proved in~\cite{Visciglia} that all $H^1$ solutions converge strongly to $0$ in $L^p (\R^N ) $, for $2< p < \frac {2N} {N-2}$, but even the rate of decay of these norms seems to be unknown. 

\end{enumerate} 
\end{rem} 

For proving Theorems~\ref{T1} and~\ref{T2}, we use the strategy of~\cite{CN}. One main ingredient is the introduction of the space $\Spa$, which is motivated by the observation that one major difficulty in studying equation~\eqref{NLS1}-\eqref{fCA1} is the lack of regularity of the nonlinearity $ |u|^\frac {2} {N} u$ (except in dimension $N=1$). However, this lack of regularity is only at $u= 0$, so it is not apparent to solutions that do not vanish. The various conditions in the definition of $\Spa$ are here to ensure a control from below of $ |u|$, provided the initial value in $\Spa$ satisfies~\eqref{fTPA11}. See~\cite[Section~1]{CN}. The other main ingredient is the application of the pseudo-conformal transformation. More precisely, given any $b>0$,  $u\in C([ 0, \infty ), \Sigma ) \cap L^\infty ((0,\infty ) \times \R^N )$ is a solution of~\eqref{NLS1} (and its equivalent formulation~\eqref{NLSI})  if and only if $v \in C([ 0, \frac {1} {b} ), \Sigma ) \cap L^\infty ((0,\frac {1} {b} ) \times \R^N )$ defined by
\begin{equation} \label{fNLS1:0}
u(t,x)=(1+bt)^{-\frac{N}{2}} e^{i \frac {b|x|^ 2 } {4(1+bt)}}v \Bigl(  \frac
{t} {1+bt} ,\frac{x} {1+bt} \Bigr) , \quad t\ge 0, \, x\in \R^N 
\end{equation}
is a solution of the nonautonomous Schr\"o\-din\-ger equation
\begin{equation}  \label{NLS2}
\begin{cases} 
iv_t + \Delta v = \lambda (1-bt)^{-1} |v|^\alpha v \\ 
v(0)= \DIb 
\end{cases} 
\end{equation} 
and its equivalent formulation
\begin{equation} \label{NLS3}
v(t)= e^{i t\Delta }\DIb - i \lambda \int _0^t (1- bs)^{-1} e^{i (t - s) \Delta } |v(s)|^\alpha v(s)\, ds
\end{equation} 
where $\DIb (x) = \DI (x)  e^{ - i \frac {b|x|^ 2 } {4}}$.
In~\cite[Theorem~1.3]{CN}, a scattering result is established for solutions of~\eqref{NLS1} with $\alpha >\frac {2} {N}$. In this case, \eqref{fNLS1:0} transforms solutions of~\eqref{NLS1} to solutions of a nonautonomous equation similar to~\eqref{NLS2}, but with $(1-bt)^{-1}$ replaced by $(1-bt)^{-\frac {4- N\alpha } {2}}$. Since $\int _0^{\frac {1} {b}} (1-bt)^{-\frac {4- N\alpha } {2}} dt = \frac {2} {b (N\alpha -2)} \to 0$ as $b\to \infty $, a solution $v$ can be constructed on the interval $[0, \frac {1} {b}) $ by a fixed point argument, provided $b$ sufficiently large. In the present case~\eqref{fCA1}, this argument cannot be applied since $(1-bt)^{-1}$  is not integrable at $\frac {1} {b}$. 
We therefore have to modify the arguments in~\cite{CN}. 
Crucial in our analysis is the elementary estimate
\begin{equation} \label{fCA19} 
\int _0^t (1-bs) ^{-1- \mu } ds = \frac {1} {b\mu } [ (1 - bt)^{ -\mu } -1]\le \frac {1} {b\mu }  (1 - bt)^{ -\mu }
\end{equation} 
for every $\mu >0$ and $t<\frac {1} {b}$. It follows that if a certain norm of $e^{i (t - s) \Delta } |v(s)|^\alpha v(s)$ is estimated by $ (1 - bs)^{ -\mu } $, then the integral in~\eqref{NLS3} is estimated in that norm by the same power $ (1 - bt)^{ -\mu } $. Concretely, this means that we can control a certain growth of $v(t)$ as $t\to \frac {1} {b}$. Technically, this is achieved by introducing an appropriate cascade of exponents. See Section~\ref{sTEQ}, and in particular Remark~\ref{eRem3}.

The rest of this paper is organized as follows. In Sections~\ref{sLIN} and~\ref{sNLE}, we establish estimates of $e^{t \Delta }$ and $ |u|^\alpha u$, which are refined versions of estimates in~\cite{CN}. 
In Section~\ref{sTEQ}, we study equation~\eqref{NLS2}. We first obtain a local existence result with a blowup alternative. Then we show that if $b$ is sufficiently large, the solution of~\eqref{NLS2} exists on $[0, \frac {1} {b} )$ and satisfies certain estimates as $t\uparrow \frac {1} {b}$. (Proposition~\ref{p3}.) This is the crux of the paper,  which requires the estimates of Sections~\ref{sLIN} and~\ref{sNLE}, as well as the introduction of an appropriate cascade of exponents.
The asymptotics of the corresponding solutions of~\eqref{NLS2} as $t\uparrow \frac {1} {b}$ is determined in Section~\ref{sASY}. Finally, the proof of Theorems~\ref{T1} and~\ref{T2} is completed in Section~\ref{sFIN}, by translating the results of  Section~\ref{sASY} in the original variables via the transformation~\eqref{fNLS1:0}.

\section{An estimate for the linear Schr\"{o}dinger equation} \label{sLIN} 

In this section,  we assume~\eqref{n2}-\eqref{n11}  (where $\alpha >0$ is arbitrary, not necessarily given by~\eqref{fCA1}), and we let $\Spa $ be defined
by~\eqref{n27}-\eqref{n28}. 
We establish estimates for the solution of the linear, nonhomogeneous  Schr\"{o}dinger equation. We recall that (see~\cite[Proposition~1]{CN}) 
\begin{equation}  \label{fCA7b1} 
 \text{$(e^{it \Delta }) _{ t\in \R }$ is a $C_0$ group on $\Spa $} 
\end{equation} 
and that there exists a constant $\CTd$  such that
\begin{equation} \label{fCA7} 
 \| e^{it \Delta } v \|_\Spa \le \CTd  \| v \|_\Spa
\end{equation} 
and
\begin{equation} \label{eLE3:12}
 \| \langle \cdot \rangle ^{\Imqd}   (e^{it \Delta }  \psi -\psi ) \| _{ L^\infty  } \le t \CTd  \| \psi  \|_\Spa 
\end{equation} 
for all $ 0\le t \le 1$ and $v\in \Spa$. 

\begin{prop} \label{eP1} 
There exists $\CTc \ge 1$ such that if $T>0$, $\DIb \in \Spa $ and $f \in C ([0, T], \Spa) $, then the solution $v$ of
\begin{equation} \label{l1}
\begin{cases} 
iv_t + \Delta v=f \\
v(0)= \DIb 
\end{cases} 
\end{equation} 
satisfies  for all $0\le t\le T$ the following estimates.
\begin{equation}  \label{n29}
 \| \langle \cdot \rangle ^\Imqd D^\beta v(t) \| _{ L^\infty  } \le  \| \DIb \| _{ \Spa } + \CTc \int _0^t (  \| v(s) \|_\Spa +  \| \langle \cdot \rangle ^\Imqd D^\beta f(s) \| _{ L^\infty  }  )\, ds
\end{equation} 
if $\left|  \beta\right|  \le 2m$,
\begin{equation}  \label{n30}
 \| \langle \cdot \rangle ^{\Imqd -\Imqs } D^\beta v(t) \| _{ L^2  } \le  \| \DIb \| _{ \Spa } + \CTc  \int _0^t (  \| v(s) \|_\Spa +  \| \langle \cdot \rangle ^{\Imqd -\Imqs } D^\beta f(s) \| _{ L^2  }  )\, ds
\end{equation} 
if $|\beta  |=\Imqp + \Imqs  +  2 \Imqu  +1 $ with $0\le \Imqp \le \Imqt +1$ and $0\le \Imqs \le \Imqd$.
\end{prop}

\begin{proof}
It follows from~\eqref{fCA7b1} that $v\in C([0, T], \Spa )$. 
We first observe that if $ |\beta |\le 2\Imqu+2$, then
\begin{equation} \label{fSC1} 
 \| \langle \cdot \rangle ^\Imqd D^\beta u \| _{ L^\infty  } \le  C \| u \| _{ \Spa } 
\end{equation} 
for all $u\in \Spa$. Indeed, if $ |\beta | \le 2\Imqu$, then~\eqref{fSC1} follows immediately from~\eqref{n28}.  
Moreover, if $2\Imqu+1 \le  |\beta | \le 2\Imqu+2$, then by Sobolev's inequality $ \| \langle \cdot \rangle ^\Imqd D^\beta u \| _{ L^\infty  } \le  C  \| \langle \cdot \rangle ^\Imqd D^\beta u \| _{ H^k } $ since $k>\frac {N} {2}$ by~\eqref{n2}. 
Applying~\cite[formula~(2.13)]{CN} (with $s=0$), we deduce that $ \| \langle \cdot \rangle ^\Imqd D^\beta u \| _{ L^\infty  } \le  C  \|  u \| _\Spa $, and~\eqref{fSC1} follows.

We now prove~\eqref{n29}. Let $\left|  \beta\right|  \le 2 \Imqu$.
Applying  $\left\langle \cdot\right\rangle ^{n}D^{\beta}$ to
equation~\eqref{l1} we obtain
\begin{equation}\label{n33}
i (  \langle \cdot \rangle ^{n}D^{\beta}v )  _{t}
=- \langle \cdot \rangle ^{n}D^{\beta}\Delta v+ \langle
\cdot \rangle ^{n}D^{\beta}f
\end{equation}
so that
\begin{equation*} 
|  \langle \cdot \rangle ^{n}D^{\beta}v |  _{t}
\le   |   \langle \cdot \rangle ^{n}D^{\beta}\Delta
v |  + |   \langle \cdot \rangle
^{n}D^{\beta}f |  .
\end{equation*} 
Integrating this last equation on $(0,t)$ with $0<t\le T$, we deduce that
\begin{equation*}
\begin{split} 
 \| \langle \cdot \rangle ^n D^\beta v(t) \| _{ L^\infty  } \le &  \| \langle \cdot \rangle ^n D^\beta \DIb \| _{ L^\infty  } \\ & +  \int _0^t (   \| \langle \cdot \rangle ^n D^\beta \Delta v(s) \| _{ L^\infty  }  +  \| \langle \cdot \rangle ^n D^\beta f(s) \| _{ L^\infty  }  )\, ds .
\end{split} 
\end{equation*} 
Inequality~\eqref{n29} follows, by using~\eqref{fSC1}. 

Next we prove \eqref{n30}. Multiplying~\eqref{n33} by $\langle \cdot \rangle ^{n - 2 \mu } D^\beta  \overline{v} $ we obtain
\begin{equation} \label{fSC2} 
\begin{split} 
\frac {1} {2}\frac {d} {dt} & \| \langle \cdot \rangle ^{n -\mu }  D^\beta v \| _{ L^2 }^2 \\ &\le  \Bigl| \Im \int  _{ \R^N  } \langle x \rangle ^{2n -2\mu }  \Delta D^\beta v  D^\beta  \overline{v}  \Bigr| +  \| \langle \cdot \rangle ^{n -\mu } D^\beta v \| _{ L^2 }  \| \langle \cdot \rangle ^{n -\mu } D^\beta f \| _{ L^2 }  \\ &=  \Bigl| \Im \int  _{ \R^N  }   D^\beta  \overline{v}   \nabla  D^\beta v \cdot  \nabla (\langle x \rangle ^{2n -2\mu } ) \Bigr| +  \| \langle \cdot \rangle ^{n -\mu } D^\beta v \| _{ L^2 }  \| \langle \cdot \rangle ^{n -\mu } D^\beta f \| _{ L^2 } 
\end{split} 
\end{equation} 
after integration by parts.
If $\mu =n$, then
\begin{equation} \label{fSC4} 
\Bigl| \Im \int  _{ \R^N  }   D^\beta  \overline{v}   \nabla  D^\beta v \cdot  \nabla (\langle x \rangle ^{2n -2\mu } ) \Bigr| =0 .
\end{equation} 
If $\mu <n$, then using the estimate $ |\nabla \langle x\rangle ^{2n-2\mu }|\le C \langle x\rangle ^{2n-2\mu -1}$ (see~\cite[formula~(A.1)]{CN}), we see that
\begin{equation} \label{fSC3} 
\Bigl| \Im \int  _{ \R^N  }   D^\beta  \overline{v}   \nabla  D^\beta v \cdot  \nabla (\langle x \rangle ^{2n -2\mu } ) \Bigr| \le C  \|  \langle \cdot \rangle ^{n-\mu -1} \nabla D^\beta v \| _{ L^2 }  \|  \langle \cdot  \rangle ^{n -\mu }  D^\beta v \| _{ L^2 }.
\end{equation} 
Since $\|  \langle \cdot  \rangle ^{n - \mu -1} \nabla D^\beta v \| _{ L^2 } \le  \| v \|_\Spa$, estimate~\eqref{n30} easily follows from~\eqref{fSC2}, \eqref{fSC4} and~\eqref{fSC3}. 
\end{proof}

\section{A nonlinear estimate} \label{sNLE} 

Throughout this section, we consider $\alpha >0$ (not necessarily given by~\eqref{fCA1}),
 we assume~\eqref{n2}-\eqref{n11}, and we let $\Spa $ be defined
by~\eqref{n27}-\eqref{n28}. 
It is proved in~\cite[Proposition~2]{CN} that there exists a constant $\CTu$ such that if $u \in \Spa$ and $\eta >0$ satisfy
\begin{equation} \label{n1}
\eta \inf_{x\in\R^N } (\langle x\rangle^{{n}}|u(x)|) \ge 1 
\end{equation}
then $  |u| ^\alpha u\in \Spa $ and 
\begin{equation} \label{n1b} 
 \| \,  |u|^\alpha u \|_\Spa \le \CTu (1 + \eta  \| u \|_\Spa )^{2J}  \| u \|_\Spa^{\alpha +1} .
\end{equation} 
Moreover, if both $u_1, u_2\in \Spa$ satisfy~\eqref{n1}, then 
\begin{equation} \label{n1c} 
\begin{split} 
 \| \,  |u_1 &|^\alpha u_1 -  |u_2|^\alpha u_2 \|_\Spa \\ & \le \CTu (1 + \eta  (\| u_1 \|_\Spa  + \| u_2 \|_\Spa ))^{2J +1}  (\| u_1 \|_\Spa  + \| u_2 \|_\Spa )^\alpha \| u_1 - u_2 \|_\Spa .
\end{split} 
\end{equation}  
We now establish a refined version of~\eqref{n1b}. 
The refinement is based on the fact that expanding $D^{\beta}(|u|^{\alpha}u)$, one obtains on the one hand a term that contains derivatives of $u$ of order $ |\beta |$ and can be estimated by $C  |u|^\alpha  |D^\beta u|$ (see~\eqref{n6}); and on the other hand terms that contain products of derivatives of $u$, all of them being of order at most $ |\beta |-1$ (see~\eqref{n8}).  
The refined version of~\eqref{n1b} is essential in our proof of Proposition~\ref{p3} below. 
(See Remark~\ref{eRem3}.) 
Given $\ell \in \N$, we set
\begin{equation} \label{fCA2} 
\|  u\| _{1, \ell} = \sup   _{ 0\le  |\beta | \le \ell  } \| \langle \cdot \rangle^{n}D^{\beta}u\| _{L^{\infty}}
\end{equation} 
\begin{equation} \label{fCA3}
\|  u\| _{2, \ell} =
\begin{cases} 
\displaystyle  \sup_{ 2{m}  +1\le  |\beta | \le \ell } \| \langle \cdot \rangle^{n}D^{\beta }u\| _{L^2  } &  \ell \ge 2m+1 \\
0 & \ell \le 2m 
\end{cases} 
\end{equation} 
and
\begin{equation} \label{fCA4} 
\|  u\| _{3, \ell } = 
\begin{cases} 
\displaystyle  \sup_{ 2{m}+3+{k} \le  |\beta | \le \ell }\| \langle \cdot \rangle^{{J}-\ell }D^{\beta }u\| _{L^ 2 } & \ell \ge 2m +3 + k \\ 0 & \ell \le 2m + 2 + k 
\end{cases} 
\end{equation} 
and we have the following estimates. 

\begin{prop} \label{p2} 
There exists a constant $\CTs \ge 1$ such that if  $u\in \Spa $ and $\eta>0$ satisfy~\eqref{n1}, then \begin{equation} \label{n3b1}
\|  \langle \cdot \rangle ^n  D^{\beta}  (|u|^{\alpha
}u) \|  _{L^{\infty}}  \le  \CTs \| u \| _{ L^\infty  }^\alpha  \| \langle \cdot\rangle ^{n}D^{\beta}u\| _{L^\infty} 
\end{equation} 
for $0 \le  |  \beta |  \le 1 $,
\begin{equation} \label{n3}
\begin{split} 
\|  \langle \cdot \rangle ^n  D^{\beta}  (|u|^{\alpha
}u) \|  _{L^{\infty}}  \le & \CTs \| u \| _{ L^\infty  }^\alpha  \| \langle \cdot\rangle ^{n}D^{\beta}u\| _{L^\infty} \\ &    + \CTs  \| u \| _{ L^\infty  }^\alpha  (1+\eta \|  u\| _{1, | \beta |  -1})^{2|  \beta|  }\|  u\|  _{1,  |  \beta|  -1} 
\end{split} 
\end{equation} 
for $2 \le  |  \beta |  \leq2{m}$,
\begin{equation} \label{n4}
\begin{split} 
\|  \langle \cdot \rangle ^{n}D^{\beta}  (|u|^{\alpha
}u) \|  _{L^ 2 } \le & \CTs \|  u \|  _{L^{\infty} }^\alpha  \|   \langle \cdot \rangle ^{n}D^{\beta}u \| 
_{L^ 2 } \\ &
+\CTs (  1+\eta\|  u\| _{1, 2m  } )  ^{2J+\alpha} (  \|  u\| _{1, {2m}}+  \| u \| _{ 2,  |\beta | -1}  ),
\end{split} 
\end{equation} 
for $2{m}+1\leq|\beta|\leq2{m}+2+{k}$, 
and
\begin{equation} \label{n5}
\begin{split} 
 \|  \langle x\rangle^{{J}-|\beta|} & D^{\beta} (|u|^{\alpha}u) \| _{L^ 2 }  \le \CTs \|  u \|  _{L^{\infty}}^{\alpha
} \|  \langle x\rangle^{{J}-|\beta|}D^{\beta}u \|  _{L^ 2 } \\
& +\CTs (  1+\eta\|  u\| _{1, {2m}} )  ^{2J+\alpha} (
\|  u\| _{1, {2m}}+\|  u\| _{2, {2{m}+2+{k}} }+ \| u \| _{ 3,  |\beta | -1} )  
\end{split} 
\end{equation} 
for $2{m}+3+{k} \leq | \beta | \leq{J}$.
\end{prop}

\begin{proof}
The case $ |\beta |\le 1$ is immediate, so we suppose $ |\beta |\ge 2$.
We observe that
\begin{equation*} 
D^{\beta}(|u|^{\alpha}u)=\sum_{\gamma+\rho=\beta}c_{\gamma,\rho}D^{\gamma
}(|u|^{\alpha})D^{\rho}u
\end{equation*} 
with the coefficients $c_{\gamma,\rho}$ given by Leibniz's rule. Since
$|u|^{\alpha}=(u\overline{u})^{\frac{\alpha}{2}}$ we see that the development
of $D^{\beta}(|u|^{\alpha}u)$ contains on the one hand the term
\begin{equation} \label{n6}
A=\left(  1+\frac{\alpha}{2}\right)  |u|^{\alpha}D^{\beta}u+\frac{\alpha}
{2}|u|^{\alpha-2}u^ 2 D^{\beta}\overline{u}, 
\end{equation}
and on the other hand, terms of the form
\begin{equation} \label{n8} 
B=|u|^{\alpha-2p}D^{\rho}u\prod_{j=1}^{p}D^{\gamma_{1,j}}uD^{\gamma_{2,j}
}\overline{u} 
\end{equation}
where
\begin{gather*} 
\gamma+\rho=\beta,\quad 1\leq p\leq|\gamma|,\quad|\gamma_{1,j}+\gamma
_{2,j}|\geq1, \\
{\displaystyle\sum _{j=0}^{p}}
(\gamma_{1,j}+\gamma_{2,j})=\gamma,\quad  | \gamma_{i,j} | \le  | \beta |-1,\text{
}i=1,2.
\end{gather*} 
It follows from~\eqref{n6} that
\begin{equation} \label{n6b} 
 | A | \le (\alpha +1)  |u|^\alpha  | D^\beta u |.
\end{equation}
Moreover, it follows from~\eqref{n1}  that $|u|^{-2p} \le \eta^{2p}\langle x\rangle^{2p{n}} $, 
so that~\eqref{n8} implies
\begin{equation} \label{n16}
|  B |  \le |u|^\alpha \eta^{2p}\langle x\rangle^{2p n } 
|D^{\rho}u|\prod_{j=1}^{p}|D^{\gamma_{1,j}}u|\,|D^{\gamma_{2,j}}u|.
\end{equation}
We begin by proving \eqref{n3}. It follows from~\eqref{n6b} that
\begin{equation}
\|  \langle \cdot \rangle ^{n}A \|  _{L^{\infty}} \le 
C \|  u \|  _{L^{\infty}}^{\alpha} \|  \langle \cdot \rangle ^{n}D^{\beta}u \|  _{L^{\infty}} . \label{n9} 
\end{equation}
Moreover, we deduce from~\eqref{n16} and~\eqref{fCA2} that
\begin{equation} \label{n17}
\| \langle \cdot \rangle^{{n}}B\| _{L^{\infty}}\le \| u \| _{ L^\infty  }^\alpha  (\eta\|  u\| 
_{1, {|  \beta |  -1}})^{2p }   \|  u\| 
_{1, {|  \beta |  -1}} . 
\end{equation}
Estimate \eqref{n3} follows from \eqref{n9} and \eqref{n17}.

Next, we prove \eqref{n4}. It follows from~\eqref{n6b} that
\begin{equation}
\left\|  \left\langle \cdot\right\rangle ^{n}A\right\|  _{L^ 2 } \le 
C\left\|  u\right\|  _{L^{\infty}}^{\alpha}\left\|  \left\langle
\cdot\right\rangle ^{n}D^{\beta}u\right\|  _{L^ 2 }. \label{n18}%
\end{equation}
Now, we estimate $ \langle x\rangle ^n B$. 
Suppose first that all the derivatives in the right-hand side of~\eqref{n16} are of order
$\leq2{m}$, then each of them is estimated by $\langle x\rangle^{-{n}}\| 
u\| _{1, 2m}$. Since also $ |u|\le \langle x\rangle   ^{-n}  \| u \| _{ 1, 2m }$, we obtain
\begin{equation}
 \langle x\rangle^{{n}}B \le \langle x\rangle ^{-n\alpha }  \| u \| _{ 1,2m }^{\alpha +1}  (\eta\|  u\| _{1, 2m})^{2p} .
\label{n19}%
\end{equation}
Moreover, $n \alpha= \frac {2n } {N} >\frac{N}{2}$ by~\eqref{n2}, so we deduce from~\eqref{n19} that
\begin{equation} \label{n21}
\| \langle \cdot \rangle^{{n}}B\| _{L^ 2 }\le  C \|  u\| _{1, 2m}^{\alpha+1} (\eta\|  u\|  _{1, 2m})^{2p} .
\end{equation}
Suppose now that one of the derivatives in the right-hand side of~\eqref{n16} is of
order greater or equal to $2{m}+1$, for instance $|\gamma_{1,1}|\geq2{m}+1$.
Note that $ |\gamma  _{i, j }| \le  |\beta |-1$, so this may only occur if $ |\beta |\ge 2m+2$.
Since the sum of all derivatives has order $|\beta|\le 2m+2+k$, we have
\begin{equation*} 
 |\beta |-  | \gamma  _{ 1,1 } |\le |\beta | - (2m+1) \le 1+k \le 1+k+n \le 2m
\end{equation*} 
by the last inequality in~\eqref{n2}. It follows that all other derivatives have order $\le 
2{m}$. Thus, \eqref{n16} and~\eqref{fCA2} yield
\begin{equation*}
\langle x\rangle ^n |  B |  \le  |u|^\alpha (\eta\|  u\| _{1, 2m})^{2p}\langle
x\rangle ^n  |D^{\gamma_{1,1} }u|. 
\end{equation*}
Since $\| \langle x\rangle^{{n}}D^{\gamma_{1,1}}u\| _{L^ 2 }\leq\| 
u\| _{2,  |\beta |-1}$ by~\eqref{fCA3}, we see that 
\begin{equation}
\| \langle \cdot \rangle^{{n}}B\| _{L^ 2 }\le    \| u \| _{ L^\infty  }^\alpha  (\eta\|  u\| _{1, 2m})^{2p} \|  u\|  _{2,  |\beta |-1}. \label{n20}%
\end{equation}
Estimates \eqref{n18}, \eqref{n21} and \eqref{n20} imply~\eqref{n4}.
(Recall that $  \| u \| _{ L^\infty  }\le  \| \langle \cdot \rangle ^n u \| _{ L^\infty  }\le  \| u \| _{ 1, 2m } $.)

Finally, we prove \eqref{n5}. It follows from~\eqref{n6b} that
\begin{equation}
\|  \langle \cdot \rangle^{{J}-|\beta|}A\|  _{L^ 2 }\leq C\| 
u \|  _{L^{\infty}}^{\alpha} \|  \langle x\rangle^{{J}-|\beta
|}D^{\beta}u \|  _{L^ 2 }. \label{n23}%
\end{equation}
We now estimate $\langle x\rangle^{{J}-|\beta|}B$. 
We first assume that all the derivatives in the right-hand side of~\eqref{n16} are of order
$\leq2{m}$. It follows that they are estimated by $\langle x\rangle^{-{n}}\| 
u\| _{1, 2m}$, and we obtain%
\[
\langle x\rangle^{{J}-|\beta|}\left|  B\right|  \leq\langle
x\rangle^{{n}}\left|  B\right|  \leq C(\eta\|  u\| _{1, 2m})^{2p}\langle x\rangle^{-\alpha n}\|  u\| _{1, 2m}^{\alpha+1}.
\]
Since $n \alpha= \frac {2n } {N} >\frac{N}{2}$ by~\eqref{n2}, we obtain
\begin{equation}
\| \langle \cdot \rangle^{{J}-|\beta|}B\| _{L^ 2 }\leq C(\eta\| 
u\| _{1, 2m})^{2p}\|  u\| _{1, 2m}^{\alpha+1}.
\label{n26}%
\end{equation}
Suppose now that one of the
derivatives in the right-hand side of~\eqref{n16} is of order $\geq2{m}+1 $,
for example $\left|  \gamma_{1,1}\right|  \geq2{m}+1$. 
Since the sum of all derivatives has order $|\beta|\le J= 2m+2+k + n$, we have
\begin{equation*} 
 |\beta |- | \gamma  _{ 1,1 } |\le |\beta | - (2m+1) \le 1+k +n \le 2m
\end{equation*} 
by the last inequality in~\eqref{n2}. It follows that all other derivatives have order $\le 
2{m}$, hence are estimated by $\langle x\rangle^{-{n}}\| 
u\| _{1, 2m}$. Therefore, \eqref{n16} yields 
\begin{equation} \label{n22b}
 |  B |  \le  |u|^\alpha (\eta\|  u\| _{1, 2m})^{2p}  |D^{\gamma_{1,1} }u|. 
\end{equation}
If $2{m}+1\le |\gamma_{1,1}|\leq2{m}+2+{k}$, we have $\| \langle \cdot \rangle
^{{J}-|\beta|}D^{\gamma_{1,1}}u\| _{L^ 2 }\leq\| \langle \cdot \rangle^{{n}%
}D^{\gamma_{1,1}}u\| _{L^ 2 }\leq\|  u\| _{2, {2{m}+2+{k}}}$,
so we deduce from~\eqref{n22b} that%
\begin{equation}
\| \langle \cdot \rangle^{{J}-|\beta|}B\| _{L^ 2 }\le (\eta\|  u\| 
_{1, 2m})^{2p} \|  u\| _{L^\infty }^{\alpha}\| 
u\| _{2, {2{m}+2+{k}}}. \label{n24}%
\end{equation}
If $2{m}+3+{k}\leq|\gamma_{1,1}|\leq|\beta|-1$, then $\| \langle
x\rangle^{{J}-|\beta|}D^{\gamma_{1,1}}u\| _{L^ 2 }\leq\| \langle
x\rangle^{{J}-|\gamma_{1,1}|}D^{\gamma_{1,1}}u\| _{L^ 2 }\leq\| 
u\| _{3,  |\beta |-1} $, and thus
\begin{equation}
\| \langle \cdot \rangle^{{J}-|\beta|}B\| _{L^ 2 }\leq(\eta\|  u\| 
_{1, 2m})^{2p}\|  u\| _{L^\infty }^{\alpha}\| 
u\| _{3,  |\beta |-1} . \label{n25}%
\end{equation}
Estimate~\eqref{n5} follows from~\eqref{n23}, \eqref{n26}, \eqref{n24} and~\eqref{n25}.
\end{proof}

\section{Local and global existence for~\eqref{NLS2}} \label{sTEQ} 

Throughout this section, we assume~\eqref{fCA1}, \eqref{n2}, \eqref{n11} and we consider $\Spa $ defined by~\eqref{n27}-\eqref{n28}.
By using the pseudo-conformal transformation~\eqref{fNLS1:0},  we transform equation~\eqref{NLS1} into the initial-value problem~\eqref{NLS2}, or its equivalent form~\eqref{NLS3}.
We begin with a local existence result for solutions of~\eqref{NLS2}, which follows from the results in~\cite{CN}.

\begin{prop} \label{p1}
Let $\lambda \in \C $ and $b\ge 0$. 
If $\DIb \in \Spa $ satisfies
\begin{equation} \label{eThm1:1}
\inf_{x\in\R^N }\langle x\rangle^{n}\left|  \DIb (x)\right|  >0, 
\end{equation}
then there exist $0< T < \frac {1} {b}$ and a unique solution $v\in  C([0,T], \Spa)$ of~$\eqref{NLS2}$ satisfying
\begin{equation} \label{fNLS12} 
\inf  _{ 0\le t\le T } \inf  _{ x\in \R^N  }  ( \langle x\rangle ^n  |v(t,x)| ) > 0 .
\end{equation} 
Moreover, $v$ can be extended on a maximal existence interval $[0, \Tma )$  with $0< \Tma \le \frac {1} {b}$ to a solution $v \in C([0, \Tma ), \Spa )$ satisfying~\eqref{fNLS12} for all $0<T<\Tma$; and if $\Tma <\frac {1} {b}$, then
\begin{equation} \label{fCA5} 
 \| v (t) \|_\Spa + \Bigl(  \inf_{x\in\R^N }\langle x\rangle^{n}\left|  v(t, x)\right|  \Bigr)^{-1} \goto _{ t\uparrow \Tma } \infty .
\end{equation} 
\end{prop}

\begin{proof} 
Given $S >0$, $f\in C([0, S], \C)$ and  $\DIb \in \Spa$ satisfying~\eqref{eThm1:1}, we consider the equation
\begin{equation} \label{fCA6} 
v(t)= e^{it \Delta } \DIb - i  \int _0^t e^{i (t -s) \Delta } f(s)  |v (s)|^\alpha v(s) \,ds .
\end{equation} 
We first observe that a local solution of~\eqref{fCA6} can be constructed by applying the method of~\cite[Proof of Proposition~3]{CN}.  Indeed, let 
\begin{align} 
\eta & \ge  2 ( \inf  _{ x\in \R^N  }  \langle  x \rangle ^\Imqd  |\DIb (x) | )^{-1} \label{fNLS8}  \\
M &\ge  1+  \CTd  \| \DIb \|_\Spa . \label{fNLS9} 
\end{align} 
Given $0< T\le S$,  set 
\begin{equation*} 
\begin{split} 
\Ens = \{ v\in C([ 0 , T ], \Spa ); & \, \| v \| _{ L^\infty ( (0,T), \Spa ) } \le M \text{ and } \\ & \eta \inf _{ x\in \R^N  } \langle x \rangle ^\Imqd   |v(t,x)| \ge 1  \text{ for }  0<t<T \}
\end{split} 
\end{equation*} 
so that $\Ens$ with the distance $\dist (u,v)=  \|u-v\|_{L^\infty ((0 ,T), \Spa )}$ is a complete metric space. Given  $v\in \Ens$, we set
\begin{equation*}
\Psi   _{ \DIb, v } (t) = e^{it \Delta } \DIb  - i \int _0^t  e^{i(t-s) \Delta } f(s)  |v|^\alpha v \, ds
\end{equation*} 
for $0\le t \le T$. 
It follows easily from~\eqref{fCA7}, \eqref{eLE3:12}, \eqref{n1b} and~\eqref{n1c} that if
\begin{equation}  \label{fNLS10} 
T (1+ \eta ) \bigl[ M +   \CTu \CTd  \| f \| _{ L^\infty (0,T) } 
 (1+ 2 \eta M)^{2 \Imqh +1} (2 M)^{\alpha  +1} \bigr] \le 1
\end{equation} 
then the map $v\mapsto \Psi  _{ \DIb, v }$ is a strict contraction $\Ens \to \Ens$; and so $\Psi  _{ \DIb, v }$ has a fixed point, which is a solution of~\eqref{fCA6} on $[0, T]$. (See~\cite[Proof of Proposition~3]{CN} for details.)

We next observe that if $\DIb \in \Spa$ satisfies~\eqref{eThm1:1}, if $0<T<\frac {1} {b}$, and if $v, w\in C( [0, T], \Spa )$ are two solutions of~\eqref{fCA6} that both satisfy~\eqref{fNLS12}, then $u=v$. This follows easily from estimates~\eqref{fCA7} and~\eqref{n1c}, and Gronwall's inequality.

We now argue as follows. We consider $\DIb \in \Spa$ satisfying~\eqref{eThm1:1}, and we first apply the local existence result for~\eqref{fCA6} with 
\begin{equation*} 
f(t)= \lambda (1-bt)^{-1}
\end{equation*} 
where $\eta$ and $M$ are chosen sufficiently large as to satisfy~\eqref{fNLS8} and~\eqref{fNLS9}, and then $0<T<\frac {1} {b}$ is chosen sufficiently small so that~\eqref{fNLS10} holds. 
This yields a solution $u\in C([0,T], \Spa )$ of~\eqref{NLS3} satisfying~\eqref{fNLS12}.  
Next, we set 
\begin{equation}  \label{fNLS13} 
\begin{split} 
\Tma =  \max \{ T\in (0,  \textstyle{ \frac {1} {b}});\, & \text{there exists a solution } \\ & v\in C([0,T], \Spa )  \text{ of }   \eqref{NLS3}  \text{ satisfying }\eqref{fNLS12}  \}.
\end{split} 
\end{equation} 
It follows that $0 < \Tma \le \frac {1} {b}$. Moreover, we deduce from the uniqueness property that there exists a solution $v\in C ( [0, \Tma) , \Spa)$ of~\eqref{NLS3} which satisfies~\eqref{fNLS12} for all $0< T < \Tma$. 
Finally, we prove the blowup alternative~\eqref{fCA5}. Assume by contradiction that $\Tma < \frac {1} {b}$, and that there exist $B>0$ and a sequence $(t_n) _{ n\ge 1 }$ such that $t_n \uparrow \Tma$ and
\begin{equation} \label{fNLS14} 
 \| v (t_n ) \|_\Spa + \Bigl(  \inf_{x\in\R^N }\langle x\rangle^{n}\left|  v(t_n ,  x)\right|  \Bigr)^{-1} \le B .
\end{equation} 
We now set $\eta =2B$ and $M= 1+ \CTd B$, so that~\eqref{fNLS8}-\eqref{fNLS9} hold with $\DIb$ replaced by $v(t_n)$, for all $n\ge 1$. We fix $\Tma < \tau  < \frac {1} {b}$, then we fix $0 < T < \tau  - \Tma $ sufficiently small so that 
\begin{equation}  \label{fNLS15} 
T (1+ \eta ) \bigl[ M +   \CTu \CTd  (1-b \tau  )^{-1}  (1+ 2 \eta M)^{2 \Imqh +1} (2 M)^{\alpha  +1} \bigr] \le 1 .
\end{equation} 
If $0\le t\le T$, then $t_n + t \le \Tma + T\le \tau $, so that $ \| (1-b (t_n + \cdot ))^{-1} \| _{ L^\infty (0,T) }\le (1-b \tau  )^{-1}$. Thus~\eqref{fNLS15} implies that~\eqref{fNLS10} is satisfied with $f (t) \equiv (1- b(t_n +t))^{-1}$ for all $n\ge 1$. It follows from the local existence result that for all $n\ge 1$ there exists $v_n \in C([0,T], \Spa)$ satisfying~\eqref{fNLS12}, which is a solution of the equation
\begin{equation*} 
v_n (t)= e^{it \Delta } v(t_n) - i  \int _0^t e^{i (t -s) \Delta } f(t_n + s)  |v_n (s)|^\alpha v_n (s) \,ds .
\end{equation*} 
Setting now
\begin{equation*} 
w_n (t) = 
\begin{cases} 
v(t) & 0\le t\le t_n \\
v_n ( t- t_n) & t_n\le t\le t_n + T
\end{cases} 
\end{equation*} 
we see that $w_n \in C([0, t_n +T], \Spa)$, that $w_n $ satisfies~\eqref{fNLS12} with $T$ replaced by $t_n +T$, and that $w_n $ is a solution of~\eqref{NLS3} on $[0, t_n+T]$. Since $t_n + T > \Tma$ for $n$ large, we obtain a contradiction with~\eqref{fNLS13}. This completes the proof. 
\end{proof} 

Our next result shows that if $\DIb \in \Spa$ satisfies~\eqref{eThm1:1} and $b$ is sufficiently large, then the corresponding solution of~\eqref{NLS2} is defined on  $[0,  \frac {1} {b})$ and satisfies certain estimates as $t \uparrow \frac {1} {b}$. We first comment on the strategy of our proof in the following remark, then we introduce the required notation and state our result in Proposition~\ref{p3}.

\begin{rem} \label{eRem3} 
We estimate derivatives of $v$, for instance $ \|  \langle \cdot \rangle ^nD^\beta v\| _{ L^\infty  }$, by a contraction argument. 
For this, we assume that 
\begin{equation}  \label{eRem3:1} 
\|  v (t)  \| _{ L^\infty  } \le  C
\end{equation} 
and
\begin{equation}  \label{eRem3:2} 
\|  \langle \cdot \rangle ^n D^\beta  v (t) \| _{ L^\infty  } \le  C (1-bt) ^{-\mu } 
\end{equation} 
and we want to recover~\eqref{eRem3:1}-\eqref{eRem3:2} through equation~\eqref{NLS2}. 
It is not too difficult to estimate $\|  v (t)  \| _{ L^\infty  }$ by using equation~\eqref{NLS2}, estimate~\eqref{eRem3:2}, and the assumption $\Im \lambda \le 0$, so we concentrate on~\eqref{eRem3:2}. 
We use Proposition~\ref{eP1}, and then we apply~\eqref{fCA19}. This yields an estimate of the form~\eqref{eRem3:2} provided $ \|  \langle \cdot \rangle ^n D^\beta (  |v|^\alpha  v )\| _{ L^\infty  } $ is also estimated by $  C (1-bt) ^{-\mu } $.  We now apply Proposition~\ref{p2} to estimate $ \|  \langle \cdot \rangle ^nD^\beta (  |v|^\alpha v )\| _{ L^\infty  } $. 
The right-hand side of~\eqref{n3} contains two terms. It follows from~\eqref{eRem3:1}-\eqref{eRem3:2} that the first term is estimated by $C (1-bt) ^{-\mu } $. Neglecting the contribution of $\eta$, the second term in~\eqref{n3} is essentially of the form $ ( \sup _{  |\gamma |\le  |\beta |-1 } \|  \langle \cdot \rangle ^n D^\gamma   v (t) \| _{ L^\infty  } )^{ 2 |\beta | +1} $.
If we assume that $ \|  \langle \cdot \rangle ^n D^\gamma   v (t) \| _{ L^\infty  }$ is estimated by $ C (1-bt) ^{-\mu } $ for $ |\gamma |\le  |\beta | -1$,  then the second term in~\eqref{n3} gives a contribution of the form $ C (1-bt) ^{-\mu ( 2| \beta |+1)} $, which is not sufficient to obtain estimate~\eqref{eRem3:2}.  
Our solution to this difficulty is to assume that derivatives of different orders are estimated by different powers of $(1-bt)$. In other words, we assume that $\mu $ in~\eqref{eRem3:2} depends on $ |\beta |$. 
Therefore, we need a cascade of exponents, which we introduce below.
\end{rem} 

Let
\begin{equation} \label{fCA8} 
0 <  \overline{\sigma }  < ( 4J+2\alpha+1)^{- J}
\end{equation} 
and set
\begin{equation} \label{fCA9} 
\sigma _j=
\begin{cases} 
0 & j=0 \\
( 4J+2\alpha+ 2)^j   \overline{\sigma }  & 1\le j\le 2m \le J 
\end{cases} 
\end{equation} 
so that 
\begin{equation} \label{fCA10} 
0= \sigma _0 <  \overline{ \sigma }  < \sigma _j< \sigma _k\le \sigma _J < 1, \quad 1\le j<k\le J.
\end{equation} 
Given $0<T < \frac {1} {b}$ and $v\in C([0,T] , \Spa )$ satisfying~\eqref{fNLS12}, we define
\begin{gather*} 
\Phi _{1, T} = \sup _{ 0\le t<  T } \sup_{ 0\le j\le 2m}  (  1-bt )  ^{\sigma _j} \|  v\| _{1, j}   \\
\Phi _{2, T} = \sup _{ 0\le t< T } \sup_{ 0 \le j\le 2m + 2+ k}  (  1-bt )  ^{\sigma _j} \|  v\| _{2, j}  \\
\Phi _{3, T}  = \sup _{ 0\le t< T }  \sup_{0 \le j\le J}  (  1-bt )  ^{\sigma _j} \|  v\| _{3, j} \\
\Phi _{4, T}  = \sup _{ 0\le t< T }  \frac { ( 1- bt )^{\sigma _1}} { \displaystyle  \inf  _{ x\in \R^N  }  \langle x \rangle ^n  | v (t, x) |} 
\end{gather*} 
where the norms $  \| \cdot \| _{j, \ell  } $ are defined by~\eqref{fCA2}--\eqref{fCA4}, and we set
\begin{gather} 
\Phi _T =  \max \{ \Phi _{1, T} , \Phi _{2, T} , \Phi _{3, T} , \} \label{fCA12:b1} \\
\Psi _T =  \max \{ \Phi _{1, T} , \Phi _{2, T} , \Phi _{3, T} , \Phi _{4, T}  \} =  \max \{ \Phi _T , \Phi _{4, T}  \}  .\label{fCA12:b2} 
\end{gather} 
Note that~\eqref{fCA2}--\eqref{fCA4}  imply
\begin{gather} 
\Phi _{1, T} = \sup _{ 0\le t<  T } \sup_{ 0\le  |\beta | \le 2m}  (  1-bt )  ^{\sigma _{ |\beta |}} \| \langle \cdot \rangle ^n D^\beta  v\| _{L^\infty }  \label{fCA12b1:1} \\
\Phi _{2, T}  = \sup _{ 0\le t< T } \sup_{ 2m + 1\le  |\beta | \le 2m + 2+ k}  (  1-bt ) ^{\sigma _{ |\beta |}} \| \langle \cdot \rangle ^n D^\beta  v\| _{L^2 } \label{fCA12b1:2} \\
\Phi _{3, T}  = \sup _{ 0\le t< T }    \sup_{ 2m + 3+ k \le  |\beta | \le J}  (  1-bt )  ^{\sigma  _{  |\beta | }}  \| \langle \cdot \rangle ^{J -  |\beta |} D^\beta  v\| _{L^2 }  \label{fCA12b1:3} .
\end{gather} 
Moreover,  one verifies easily  that  
\begin{gather}  
\Phi _T \le  \| v \| _{ L^\infty  ((0,T), \Spa )}  \label{fCA13} \\
\Phi _T \ge  \| \langle \cdot \rangle ^n v  \| _{ L^\infty  ((0,T) \times \R^N )} + \frac {1} {\CTq} (1-b T )^{\sigma _J}  \| v  \| _{L^\infty ((0,T), \Spa )} \label{fCA14} 
\end{gather} 
where the constant $\CTq \ge 1$ is independent of $T$.

\begin{prop}  \label{p3}
Suppose $\Im \lambda \le 0$.  Given any $K>0$, there exists $b_0>1$ such that if $\DIb \in \Spa $ satisfies
\begin{equation} \label{fWTd}
 \| \DIb \|_\Spa +  \Bigl( \inf  _{ x\in \R^N  }  \langle x \rangle ^n  | \DIb (x) | \Bigr)^{-1} \le K
\end{equation} 
then for every $b\ge b_0$ the corresponding solution $v \in C ( [0, \Tma ), \Spa )$ of~\eqref{NLS3} given by Proposition~$\ref{p1}$,  satisfies $\Tma = \frac {1} {b}$ and 
\begin{equation} \label{fWT1}
 \sup  _{ 0< T < \frac {1} {b} } \Psi _T  \le  4 K
\end{equation}
where $\Psi _T $ is defined by~\eqref{fCA12:b2}. 
\end{prop}

\begin{proof}
Since $v\in C([0, \Tma), \Spa )$, we see that $  \| v \| _{ L^\infty ((0,T), \Spa) }  \to  \| \DIb \|_\Spa $ as $T\downarrow 0$. Therefore,  it follows from~\eqref{fCA13} and~\eqref{fWTd}  that $  \| v \|_T \le 2K$ if $T\in (0, \Tma)$ is sufficiently small, where $K$ is given by~\eqref{fWTd}. Moreover, from~\eqref{fWTd} and the property $v\in C([0, \Tma), \Spa )$, we  deduce that 
\begin{equation*} 
\sup _{ 0<t < T} \Bigl( \inf  _{ x\in \R^N  }  \langle \cdot \rangle ^n  | v (t, x) | \Bigr)^{-1}  \le  2 K 
\end{equation*} 
if $T\in (0, \Tma)$ is sufficiently small. 
Therefore, if we set
\begin{equation} \label{fCAb17}
T^\star = \sup  \{   0<T< \Tma;   \Psi _T \le  4 K \} 
\end{equation} 
then we see that $0 < T^\star \le \Tma$.
We claim that if $b$ is sufficiently large, then 
\begin{equation} \label{fFIN2} 
T^\star = \Tma .
\end{equation} 
Assuming~\eqref{fFIN2}, the conclusion of the theorem follows. Indeed, \eqref{fCA12:b2} and~\eqref{fCA14} imply that 
\begin{equation} \label{fFIN4} 
\Psi _T \ge (1- bT)^{\sigma _J} \max \Bigl\{ \sup  _{ 0< t< T }     \Bigl(  \inf  _{ x\in \R^N  }  \langle x \rangle ^n  | v (t, x) |  \Bigr)^{-1}, \frac {1} {\CTq}  \| v \| _{ L^\infty ((0,T), \Spa ) } \Bigr\}
\end{equation} 
If~\eqref{fFIN2} holds and $\Tma < \frac {1} {b}$, then it follows from~\eqref{fFIN4}  that 
\begin{equation*} 
\limsup _{ t\uparrow \Tma } \| v (t) \|_\Spa + \Bigl(  \inf_{x\in\R^N }\langle x\rangle^{n}\left|  v(t, x)\right|  \Bigr)^{-1} \le 4K (1 + \CTq) (1 -b \Tma)^{- \sigma _J} <\infty 
\end{equation*} 
which contradicts the blowup alternative~\eqref{fCA5}. Therefore, we have $T^\star = \Tma = \frac {1} {b}$, from which the desired conclusion easily follows. 

We now prove the claim~\eqref{fFIN2}, and we assume by contradiction that 
\begin{equation} \label{fFIN5} 
T^\star < \Tma .
\end{equation} 
It easily follows from~\eqref{fCAb17} and~\eqref{fFIN5} that
\begin{equation} \label{fFIN6} 
\Psi  _{ T^\star } =4K .
\end{equation} 
We will use the elementary estimate~\eqref{fCA19}, as well as the following consequence of~\eqref{fCA14} and~\eqref{fFIN6}. 
\begin{equation}  \label{fCAb2} 
\int _0^t  \| v(s) \|_\Spa \le 4K \CTq  \int _0^t (1-bs)^{-\sigma _J} \le  \frac { 4K \CTq } { b (1-\sigma _J) }  .
\end{equation} 
Next, we set
\begin{equation} \label{fFIN7b1} 
\eta (t)= 4K (1 -bt )^{- \sigma _1} 
\end{equation} 
so that by~\eqref{fFIN6} 
\begin{equation} \label{fFIN7} 
  \eta (t)  \inf  _{ x\in \R^N  }  \langle x \rangle ^n  | v (t, x) |   \ge 1
\end{equation} 
for all $0\le t\le T^\star$. Moreover, it follows from~\eqref{fFIN6} that for all $0\le t\le  T^\star$
\begin{equation} \label{fQQC} 
 \| v(t) \| _{ p , q } \le 4K (1 -bt )^{- \sigma _{q } } \quad  \text{if}\quad  
\begin{cases} 
0\le q \le 2 m & p=1  \\
 0\le q \le 2 m +2 + k & p=2 \\
 0\le q \le J & p=3 .
\end{cases} 
\end{equation} 
If $0\le j \le 2 m$, then by~\eqref{fFIN7b1} and~\eqref{fQQC} yield
\begin{equation} \label{fFIN7b4} 
1+ \eta (t)  \| v (t) \| _{ 1, j } \le 1+  (4K)^2 (1 -bt )^{- \sigma _1- \sigma _j} \le 2  (4K)^2 (1 -bt )^{- \sigma _1- \sigma _j} 
\end{equation}  
since $K\ge 1$.  Consider now
\begin{gather*} 
0\le \rho \le 2J+\alpha \\
0\le j\le 2m
\end{gather*} 
and $\ell, q$ such that
\begin{equation*} 
\begin{cases} 
\max\{2, j+1 \} \le \ell  \le 2 m & p=1  \\
\max\{2, j+1 \} \le \ell \le 2 m +2 + k & p=2 \\
\max\{2, j+1 \} \le \ell \le J & p=3 .
\end{cases} 
\end{equation*} 
Using the properties $\sigma _1\le \sigma  _{ \ell -1 }$, $\sigma _j \le \sigma  _{ \ell -1 }$, and $(4J + 2 \alpha +1) \sigma  _{ \ell -1 } \le \sigma _\ell$ (see~\eqref{fCA9}), we deduce from~\eqref{fQQC} (with $q =\ell -1$) and~\eqref{fFIN7b4}  that
\begin{equation} \label{fFIN7b2} 
\begin{split} 
( 1+ \eta (t)  \| v (t) \| _{ 1, j } ) ^\rho  \| v(t) \| _{ p , \ell -1} & \\  \le 2 ^{2J+ \alpha } & (4K)^{4J + 2\alpha +1} (1 -bt )^{ - (2J + \alpha ) (\sigma _1 + \sigma _j ) - \sigma  _{ \ell -1 } } \\ & \le  (8 K)^{4J + 2\alpha +1} (1 -bt )^{- \sigma _\ell} 
\end{split} 
\end{equation} 
for all $0\le t\le T^\star$.

We now estimate $\Phi _{4, T^\star} $.
It follows from~\eqref{NLS2} that (recall that $ |v|>0$ on $[0, T^\star ] \times \R^N $)
\begin{equation} \label{n46}
|  v |  _t = L + \Im \lambda (  1-bt ) ^{-1}|v|^{\alpha+1} 
\end{equation}
where
\begin{equation} \label{fDFL} 
L(t, x) = i\frac{ (  \overline{v}\Delta v-v\Delta \overline{v} )  }{2|v|} 
=  -  \frac {\Im  (  \overline{v} \Delta v )} { |v|} .
\end{equation} 
It follows from~\eqref{n46} that
\begin{equation} \label{n46:b1}
- \frac {1} {\alpha } \frac {\partial } {\partial t}  (  |v|^{ -\alpha } ) =    |v|^{ -\alpha -1} L    + \Im \lambda (  1-bt ) ^{ -1 }   .
\end{equation}
Setting
\begin{equation*} 
w (t, x) = \langle x\rangle ^n  |v (t,x) |
\end{equation*} 
we deduce from~\eqref{n46:b1} that
\begin{equation} \label{fCAc2}
- \frac {1} {\alpha } \frac {\partial } {\partial t}  ( w^{ - \alpha } ) =    \langle x\rangle ^n  w^{ - \alpha -1}  L   + \Im \lambda (  1-bt ) ^{ -1 } \langle x\rangle ^{-n\alpha }   .
\end{equation}
We note that for $0\le t< T^\star $
\begin{equation} \label{fCAc3}
  \langle x\rangle ^n   |L  | \le \| \langle \cdot \rangle ^n    \Delta v \| _{ L^\infty  } \le ( 1- bt )^{- \sigma _2}  \Psi  _{ T^\star  } \le 4K ( 1- bt )^{- \sigma _2} 
\end{equation} 
by~\eqref{fFIN6}. Integrating~\eqref{fCAc2}  in $t$, and applying~\eqref{fCAc3},   we obtain
\begin{equation*} 
\frac {1} {w(t,x)^\alpha }  \le \frac {1} {w( 0 ,x)^\alpha } + 4 \alpha K \int _0^t \frac {ds} {(1-bs) ^{ \sigma _2} w(s, x)^{\alpha +1}} 
+  \alpha  | \Im \lambda  |\int _0^t \frac {ds} { 1-bs } .
\end{equation*} 
Since $\frac {1} {w( 0 ,x)} \le K$ by~\eqref{fWTd} and $\frac {1} {w( t ,x)} \le 4K (1-b t ) ^{-\sigma _1}$ by~\eqref{fFIN7b1}-\eqref{fFIN7},  the above estimate implies
\begin{equation} \label{fCAc4}
\begin{split} 
\frac {1} {w(t,x)^\alpha } & \le  K^\alpha  + \alpha (4K )^{ \alpha +2} \int _0^t \frac {ds} {(1-bs) ^{ \sigma _2+ (\alpha +1) \sigma _1}  } 
+ \frac { \alpha   | \Im \lambda  | } {b }  | \log (1 - bt ) | .
\end{split} 
\end{equation} 
Note that $ \sigma _2+ (\alpha +1) \sigma _1 \le 2 \sigma _2 \le \sigma _J$ by~\eqref{fCA9}, so that~\eqref{fCAc4} yields
\begin{equation*}
\begin{split} 
\frac {1} {w(t,x)^\alpha } &  \le  K^\alpha  + \frac { \alpha (4K )^{ \alpha +2}  } {b (1- \sigma _J )} + \frac {  | \Im \lambda  | } {b }  | \log (1 - bt ) | ,
\end{split} 
\end{equation*} 
from which it follows that
\begin{equation}  \label{fCAc5}
\begin{split} 
\Phi _{4, T^\star}  & \le \sup _{ 0\le t< T^\star }    \Bigl[  ( 1- bt )^{\sigma _1}   \Bigl( K^\alpha  + \frac {\alpha (4K )^{ \alpha +2}  } {b (1- \sigma _J )} + \frac { \alpha  | \Im \lambda  | } {b }  | \log (1 - bt ) | \Bigr)^{\frac {1} {\alpha }} \Bigr]  \\ & \le    \Bigl( K^\alpha  + \frac { \alpha (4K )^{ \alpha +2}  } {b (1- \sigma _J )} + \frac { \alpha  | \Im \lambda  | } {b }  \sup _{ 0\le t< \frac {1} {b} }   [  ( 1- bt )^{\alpha \sigma _1}   | \log (1 - bt ) | ] \Bigr)^{\frac {1} {\alpha }} \\ & =    \Bigl( K^\alpha  + \frac { \alpha (4K )^{ \alpha +2}  } {b (1- \sigma _J )} + \frac { \alpha  | \Im \lambda  | } {b }  \sup _{ 0\le t< 1 }   [  t^{\alpha \sigma _1}   | \log t | ] \Bigr)^{\frac {1} {\alpha }}  .
\end{split} 
\end{equation} 
We next estimate $  \| \langle \cdot \rangle ^n v\| _{ L^\infty  } $. It follows from~\eqref{n46} that
\begin{equation} \label{fCAb9}
|  v |  _t \le   | \Delta v | .
\end{equation}
Note that by~\eqref{fCAb2}
\begin{equation} \label{fCA18} 
\int _0^t  \|  \langle x\rangle ^n \Delta v (s) \| _{ L^\infty  } \le \int _0^t  \| v(s) \|_\Spa \le \frac {4 K \CTq } { b (1-\sigma _J) }   .
\end{equation} 
Applying~\eqref{fCAb9}, \eqref{fCA18} and~\eqref{fWTd},  we obtain
\begin{equation} \label{n49}
 \| \langle \cdot \rangle ^n v(t) \| _{ L^\infty  } \le  \| \langle \cdot \rangle ^n \DIb \| _{ L^\infty  } + \int _0^t   \| \langle \cdot \rangle ^n \Delta v(s) \| _{ L^\infty  } \le  K +  \frac { 4 K \CTq  } { b (1-\sigma _J) }    .
\end{equation} 
We now estimate $ \| \langle \cdot \rangle ^n D^\beta v \| _{ L^\infty  }$ for $1\le  |\beta | \le 2m$, and we use the estimates of Propositions~\ref{eP1} and~\ref{p2}. 
Applying~\eqref{n29}, \eqref{fFIN7}, \eqref{n3b1}   and \eqref{n3}, we deduce that
\begin{equation}  \label{n36}
\begin{split} 
 \| \langle \cdot \rangle ^n D^\beta v \| _{ L^\infty  } & \le  \| \DIb \| _\Spa + \CTc  \int _0^t (  \|v \|_\Spa +  |\lambda | (1- bs )^{ -1}  \| \langle \cdot \rangle ^n D^\beta ( |v  |^\alpha v ) \| _{ L^\infty  }) \\
  \le  \| \DIb \| &_\Spa + \CTc  \int _0^t    \|v \|_\Spa +  |\lambda | \CTc \CTs  \int _0^t (1- bs)^{-1}  \| v\| _{ L^\infty  }^\alpha   \| \langle \cdot \rangle ^n D^\beta v \| _{ L^\infty  }  \\
   + \kappa  |\lambda | & \CTc \CTs  \int _0^t   (1- bs)^{-1}  \| v \| _{ L^\infty  }^\alpha  (1+ \eta (s) \|  v \| _{1, | \beta |  -1})^{2|  \beta|  }\|  v \|  _{1,  |  \beta|  -1} 
\end{split} 
\end{equation} 
with $\kappa =0$ if $ |\beta |=1$ and $\kappa =1$ if $ |\beta | \ge 2$.
Moreover, $  \|  \langle \cdot \rangle ^{n}D^{\beta}v\| _{L^\infty} \le \|  v\| _{1,   |\beta |} $, so that by~\eqref{fFIN6} 
\begin{equation} \label{fCA20} 
\|  v\| _{L^{\infty}}^\alpha \|  \langle \cdot \rangle ^{n}D^{\beta}v\| _{L^\infty}
\le   (  1-bs )  ^{-\sigma _{ |  \beta |  }}(4K)^{\alpha +1} .
\end{equation}
We deduce from~\eqref{fCA20} and~\eqref{fCA19} that
\begin{equation} \label{n39} 
 \int _0^t (1- bs)^{-1}  \| v\| _{ L^\infty  }^\alpha   \| \langle \cdot \rangle ^n D^\beta v \| _{ L^\infty  } \le \frac { (4K)^{\alpha +1} } {b \sigma  _{  |\beta | }} (1-bt )^{ -\sigma  _{  |\beta | }}.
\end{equation} 
Next, assuming $ |\beta | \ge 2$, we apply~\eqref{fFIN7b2} with $j=  |\beta |-1$, $\rho = 2 |\beta |$, $p= 1$ and $\ell =  |\beta |$, to obtain
\begin{equation*}
( 1+ \eta (s)  \| v ( s ) \| _{ 1,  |\beta |j-1 } ) ^{2 |\beta |}  \| v( s ) \| _{ 1 ,  |\beta | -1}  \le  (8 K)^{4J + 2\alpha +1} (1 -bs )^{- \sigma _{ |\beta |}} 
\end{equation*} 
so that
\begin{equation*}
 \| v (s) \| _{ L^\infty  }^\alpha ( 1+ \eta (s)  \| v ( s ) \| _{ 1,  |\beta |j-1 } ) ^{2 |\beta |}  \| v( s ) \| _{ 1 ,  |\beta | -1}  \le  (8 K)^{4J + 3\alpha +1} (1 -bs )^{- \sigma _{ |\beta |}} .
\end{equation*} 
Applying~\eqref{fCA19}, we deduce that
\begin{equation}  \label{n38}
\begin{split} 
 \int _0^t   (1- bs)^{-1}  \| v \| _{ L^\infty  }^\alpha   (1+ \eta (s) \|  v \| _{1, | \beta |  -1})^{2|  \beta|  } & \|  v \|  _{1,  |  \beta|  -1}  \\ & \le \frac {  (8 K)^{4J + 3\alpha +1} } {b \sigma  _{  |\beta | }} (  1-bt )  ^{ - \sigma  _{  |\beta | } } .
\end{split} 
\end{equation} 
It follows from~\eqref{n36},  \eqref{fWTd}, \eqref{fCAb2}, \eqref{n39} and~\eqref{n38} that
\begin{equation} \label{n40}
 \| \langle \cdot \rangle ^n D^\beta v \| _{ L^\infty   } \le K + \frac {4 K \CTq \CTc } { b (1-\sigma _J) } +  \frac { 2  |\lambda | \CTc \CTs  (8K)^{4J + 3\alpha  +1}  } {b \sigma  _{  |\beta | }}    (1-bt )^{ -\sigma  _{  |\beta | }}.
\end{equation} 
We next estimate $ \| \langle \cdot \rangle ^n D^\beta v \| _{ L^2  }$ for $2{m}+1\le |  \beta | \le 2m +2+ k$.
Estimates~\eqref{n30} (with $\mu =0$ and $\nu =  |\beta |- 2m -1$), \eqref{fFIN7}  and \eqref{n4} imply
\begin{equation}   \label{n41}
\begin{split} 
 \| \langle \cdot \rangle ^n D^\beta v \| _{ L^2 } & \le  \| \DIb \| _\Spa + \CTc  \int _0^t (  \|v \|_\Spa +  |\lambda | (1- bs )^{ -1}  \| \langle \cdot \rangle ^n D^\beta ( |v  |^\alpha v ) \| _{ L^2  }) \\
  \le  \| \DIb \| &_\Spa + \CTc  \int _0^t    \|v \|_\Spa +  |\lambda | \CTc \CTs \int _0^t (1- bs)^{-1}  \| v\| _{ L^\infty  }^\alpha   \| \langle \cdot \rangle ^n D^\beta v \| _{ L^2  }  \\
   + \CTc  \CTs & |\lambda |  \int _0^t (1- bs)^{-1}   (1+\eta\|  v \| _{1, 2m })^{2J+\alpha  } (  \| v \| _{ 1, 2m } + \|  v \|  _{2,  |  \beta|  -1}  ).
\end{split} 
\end{equation} 
We have $  \|  \langle \cdot \rangle ^{n}D^{\beta}v\| _{L^2} \le \|  v\| _{2,   |\beta |} $, so that by~\eqref{fFIN6} 
\begin{equation*} 
\|  v\| _{L^{\infty}}^\alpha \|  \langle \cdot \rangle ^{n}D^{\beta}v\| _{L^2}
\le  (4K) ^{\alpha+1} (  1-bs )  ^{-\sigma _{ |  \beta |  }} .
\end{equation*}
Applying~\eqref{fCA19}, we deduce that
\begin{equation} \label{fCA22} 
 \int _0^t (1- bs)^{-1}  \| v\| _{ L^\infty  }^\alpha   \| \langle \cdot \rangle ^n D^\beta v \| _{ L^2  } \le \frac { (4K)^{\alpha +1}} {b \sigma  _{  |\beta | }} (1-bt )^{ -\sigma  _{  |\beta | }}.
\end{equation} 
Next, we have by applying~\eqref{fFIN7b2} with $j= 2m$, $\rho = 2J + \alpha $, and successively $p= 1$ and $\ell = 2m+1$, then $p= 2$ and $\ell =  |\beta |$
\begin{equation*} 
( 1+  \eta\|  v \| _{1, 2m })^{2J+\alpha  } (  \| v \| _{ 1, 2m } + \|  v \|  _{2,  |  \beta|  -1}  )  \le      2 (8 K)^{4J + 2\alpha +1} (1 -bs )^{- \sigma _{  |\beta | }}  .
\end{equation*} 
It then follows from~\eqref{fCA19} that
\begin{equation}  \label{n43}
\begin{split} 
  \int _0^t (1- bs)^{-1}   (1+\eta\|  v \|  _{1, 2m })^{2J+\alpha  }  (  \| v \| _{ 1, 2m } + \|  v \|  _{2,  |  \beta|  -1}  ) & \\  \le \frac { 2 (8 K)^{4J + 2\alpha +1} } {b \sigma  _{  |\beta | }} &   (  1-bt )  ^{ - \sigma  _{  |\beta | } } .
\end{split} 
\end{equation} 
Applying~\eqref{fWTd}, \eqref{fCAb2}, \eqref{fCA22} and~\eqref{n43}, we deduce from~\eqref{n41} that
\begin{equation} \label{n40b1}
 \| \langle \cdot \rangle ^n  D^\beta v \| _{ L^2  }  \le  K + \frac { 4K \CTq \CTc } { b (1-\sigma _J) }  +  \frac { 3  |\lambda | \CTc \CTs (8 K)^{4J + 2\alpha +1}} {b \sigma  _{  |\beta | }}   (1-bt )^{ -\sigma  _{  |\beta | }}.
\end{equation} 
Now, we estimate  $ \| \langle \cdot \rangle ^{J -  |\beta | } D^\beta v \| _{ L^2  }$ for $ {m}+ 3+{k} \le  |\beta | \le J$.
It follows from~\eqref{n30} (with $\mu =-   |\beta | + n - J$ and $\nu = k+1$), \eqref{fFIN7},  and \eqref{n5} that
\begin{equation}   \label{fCA24}
\begin{split} 
 \| \langle \cdot \rangle ^{J -  |\beta |}  D^\beta v \| _{ L^2 }  &
  \le  \| \DIb \| _\Spa  + \CTc  \int _0^t    \|v \|_\Spa \\  & +  |\lambda | \CTc \CTs \Bigl[ \int _0^t (1- bs)^{-1}  \| v\| _{ L^\infty  }^\alpha   \| \langle \cdot \rangle ^{J -  |\beta |}  D^\beta v \| _{ L^2  }  \\
 +   \int _0^t (1- bs)^{-1}  & (1+\eta  \|  v  \| _{1, 2m })^{2J+\alpha  } (  \| v \| _{ 1, 2m } + \|  v \|  _{2, 2m + 2 + k} +  \| v \| _{ 3,  |\beta | -1} ) \Bigr].
\end{split} 
\end{equation} 
We have $  \|  \langle \cdot \rangle ^{J-  |\beta |}D^{\beta}v\| _{L^2} \le \|  v\| _{3,   |\beta |} $, hence
\begin{equation} \label{fCA25} 
\|  v\| _{L^{\infty}}^\alpha \|  \langle \cdot \rangle ^{J - |\beta |}D^{\beta}v\| _{L^2}
\le (4 K )^{\alpha+1} (  1-bt )  ^{-\sigma _{ |  \beta |  }} 
\end{equation}
by~\eqref{fFIN6}.  
Applying~\eqref{fCA19}, we obtain
\begin{equation} \label{fCA26} 
 \int _0^t (1- bs)^{-1}  \| v\| _{ L^\infty  }^\alpha   \| \langle \cdot \rangle ^{J - |\beta |} D^\beta v \| _{ L^2  } \le \frac { (4K)^{ \alpha +1}} {b \sigma  _{  |\beta | }} (1-bt )^{ -\sigma  _{  |\beta | }}.
\end{equation} 
Next, we apply~\eqref{fFIN7b2} with $j= 2m$, $\rho = 2J + \alpha $, and successively $p= 1$ and $\ell = 2m+1$, then $p= 2$ and $\ell = 2m + 3 + k$, then $p= 3$ and $\ell =  |\beta |$

\begin{equation*} 
\begin{split} 
( 1+  \eta\|  v \| _{1, 2m })^{2J+\alpha  } (   \| v \| _{ 1, 2m } + \|  v \|  _{2, 2m + 2 + k} +&  \| v \| _{ 3,  |\beta | -1}  ) \\  &\le   3 (8 K)^{4J + 2\alpha +1} (1 -b s  )^{- \sigma _\ell} .
\end{split} 
\end{equation*} 
Therefore, we deduce from~\eqref{fCA19} that
\begin{equation}  \label{n44}
\begin{split} 
  \int _0^t (1- bs)^{-1}   (1+\eta\|  v \|  &_{1, 2m })^{2J+\alpha  }  (  \| v \| _{ 1, 2m } + \|  v \|  _{2, 2m + 2 + k} +  \| v \| _{ 3,  |\beta | -1} ) \\  \le \frac {  3 (8 K)^{4J + 2\alpha +1}  } {b \sigma  _{  |\beta | }} &  (  1-bt )  ^{ - \sigma  _{  |\beta | } } 
\end{split} 
\end{equation} 
Applying~\eqref{fWTd}, \eqref{fCAb2}, \eqref{fCA26} and~\eqref{n44}, we deduce from~\eqref{fCA24} that
\begin{equation} \label{n40b2}
 \| \langle \cdot \rangle ^{ J-  |\beta |}  D^\beta v \| _{ L^2  } \le  K + \frac { 4 K \CTq \CTc } { b (1-\sigma _J) } \ +  \frac { 4 | \lambda |  \CTc \CTs (8 K)^{4J + 2\alpha +1}  } {b \sigma  _{  |\beta | }}  (1-bt )^{ -\sigma  _{  |\beta | }}.
\end{equation} 
It follows from~\eqref{fCA12b1:1}--\eqref{fCA12b1:3}, \eqref{n49}, \eqref{n40}, \eqref{n40b1}, and~\eqref{n40b2}   that
\begin{equation} \label{fCAb11} 
\Phi  _{ T^\star  } \le  K + \frac { 4 K \CTq \CTc } { b (1-\sigma _J) }  +   \frac { 4  | \lambda | \CTc \CTs (8 K)^{4J + 3\alpha +1}  } {b \sigma  _{  |\beta | }}  . 
\end{equation} 
Finally, we assume that $b_0$ is sufficiently large so that
\begin{equation}  \label{fCAb12} 
 \frac { 4K \CTq \CTc } { b_0 (1-\sigma _J) }  +   \frac { 4 | \lambda |  \CTc \CTs (8 K)^{4J + 3\alpha +1}  } {b_0 \sigma  _{  |\beta | }}  \le K
\end{equation} 
and
\begin{equation} \label{fCAb14} 
  \Bigl( K^\alpha  + \frac { \alpha (4K )^{ \alpha +2}  } {b_0 (1- \sigma _J )} + \frac { \alpha  | \Im \lambda  | } {b_0 }  \sup _{ 0\le t< 1 }   [  t^{\alpha \sigma _1}   | \log t | ] \Bigr)^{\frac {1} {\alpha }}  \le 2K.
\end{equation} 
We deduce from~\eqref{fCAc5} and~\eqref{fCAb14}, that if $b\ge b_0$, then
\begin{equation} \label{fCAb20} 
\Phi _{4, T^\star}  \le  2K.
\end{equation} 
Moreover, we deduce from~\eqref{fCAb11} and~\eqref{fCAb12}, that if $b\ge b_0$, then
\begin{equation} \label{fCAb21} 
\Phi _{ T^\star } \le 2K .
\end{equation} 
Inequalities~\eqref{fCAb20} and~\eqref{fCAb21} yield
$\Psi _{ T^\star }   \le 2 K $, 
which contradicts~\eqref{fFIN6}, thus  completing the proof.
\end{proof}

\begin{rem} \label{eRem4} 
Note that the only place in the proof of Proposition~\ref{p3} where we use the assumption $\Im \lambda \le 0$ is estimate~\eqref{fCAb9}.
Yet, the conclusion of Proposition~\ref{p3} fails if $\Im \lambda >0$. More precisely, if $\DIb \in \Spa$ satisfies~\eqref{fWTd} and $b>0$, then there is no solution $v \in C ( [0, \frac {1} {b} ), \Spa )$ of~\eqref{NLS3} satisfying~\eqref{fWT1}. Indeed, suppose that $v \in C ( [0, \frac {1} {b} ), \Spa )$ satisfies~\eqref{NLS3} and~\eqref{fWT1}. Applying identity~\eqref{fCAc2} with $x=0$ and integrating in $t$ yields
\begin{equation*}
0 \ge  -  w (  t, 0)^{ - \alpha }   = -  w ( 0,0)^{ - \alpha } + \alpha   \int _0^t  w(s, 0)^{ - \alpha -1}  L (s, 0)\, ds   + \frac {\alpha } {b}  \Im \lambda  | \log (  1-bt ) | .
\end{equation*}
Since the integral on the right-hand side of the above inequality is bounded as $t\uparrow \frac {1} {b}$ by~\eqref{fWT1}, we obtain a contradiction by letting $t\uparrow \frac {1} {b}$.

\end{rem} 

\section{Asymptotics  for~\eqref{NLS2}} \label{sASY} 

We now turn to the study of the asymptotic of the solution $v$ as
$t \rightarrow \frac{1} {b}$. We prove the following:

\begin{prop} \label{p4}
Suppose $\Im \lambda \le 0$.
 Assume~\eqref{fCA1}, \eqref{n2}, \eqref{n11} and let
$\Spa $ be defined by~\eqref{n27}-\eqref{n28}. Let $K\ge 1$, and let $b_0$ be given by Proposition~$\ref{p3}$. Suppose $b\ge b_0$, let $ \DIb
\in \Spa $  satisfy~\eqref{fWTd}, and let $v\in C([0, \frac {1} {b}), \Spa )$ be the solution of~\eqref{NLS2} given by Proposition~$\ref{p3}$.
There exists  $b_{1} \ge  b_{0}$ such that if $b \ge  b_{1}$, then there exist $f_{0}, w _0 \in L^{\infty},$ with $f_0$ real valued,  $ \|
f_{0}  \| _{L^{\infty}} \le \frac{1}{2}$, $w_0 \not \equiv 0$ and $ \langle \cdot \rangle ^n w_0\in L^\infty  (\R^N )  $ such that 
\begin{equation} \label{n63}
\| \langle \cdot \rangle ^n ( v (  t , \cdot )  - w_0 (\cdot ) \psi (t, \cdot ) e^{- i \theta (t, \cdot )} ) \|
_{L^{\infty}} \le  C(1-bt)^{  1-\sigma_J  }
\end{equation}
for all $0\le t<\frac {1} {b}$, where
\begin{equation} \label{fTCA1} 
\psi (  t,x )  =  \left(  \frac{ 1+ f_0 (x) } {1+ f_0 (x)+ \frac {\alpha |\Im \lambda |} {b}  |
\DIb (x)  | ^{\alpha} | \log(1-bt)|  } \right)  ^{\frac {1} {\alpha }}
\end{equation} 
and
\begin{equation} \label{fTCA2} 
\theta (  t,x )  =\frac {\Re \lambda} {b} \int_{0}^{ | \DIb (x) | ^{\alpha}   | \log (1-bt )| } \frac{ d\tau } {1+ f_0 (x)+ \tau  \frac {\alpha |\Im \lambda |} {b}    }  .
\end{equation}
In addition, if $\Im \lambda =0$, then 
\begin{equation} \label{fSPPz} 
\psi (t, x) \equiv 1  \text{ and } \theta   (t, x) = \frac {\lambda } {b}  | w_0 (  x)  |^\alpha | \log (1- bt) | .
\end{equation}
Furthermore,
\begin{equation} \label{fp4:1} 
 \| v (t) \| _{ L^\infty  }\goto  _{ t \uparrow \frac {1} {b} }  \| w_0 \| _{ L^\infty  }
\end{equation} 
if $ \Im \lambda =0 $ and
\begin{equation} \label{fp4:2} 
  | \log (1-bt) |^{\frac {N} {2}}  \| v (t) \| _{ L^\infty  }\goto  _{ t \uparrow \frac {1} {b} }  \Bigl(  \frac {b} {\alpha  | \Im \lambda |}   \Bigr)^{\frac {N} {2 }}
\end{equation} 
if $ \Im \lambda <0 $.
\end{prop}

\begin{proof}
We first determine the asymptotic behavior of $ |v|$.
Integrating equation~(\ref{n46:b1}) on
$(0,t)$ with $0\le  t<\frac {1} {b}$, we obtain
\begin{equation} \label{fCAb26}
\frac {1} { \alpha  |v| ^\alpha  } = \frac {1} { \alpha  | \DIb | ^\alpha } + \frac { |\Im \lambda |} {b}  | \log (1-bt) | -\int_0 
^t  |v|^{-\alpha -1} L ,
\end{equation}
where $L$ is defined by~\eqref{fDFL}, so that
\begin{equation} \label{n110}
 |v| ^\alpha  = \frac {  | \DIb | ^\alpha } {1 + f + \frac {\alpha  |\Im \lambda |} {b}  | \DIb | ^\alpha  | \log (1-bt) | }
 \end{equation}
 with 
 \begin{equation} \label{n110:b1}
 f(t,x)= - \alpha  \int _0^t  | \DIb (x) | ^\alpha  |v (s, x) |^{-\alpha -1} L(s, x) \, ds .
 \end{equation} 
Since $ \| \DIb \|_\Spa \le K$, we have $\langle x\rangle^n  | \DIb (x) |\le K$. 
Moreover,
\begin{gather*} 
( \langle x \rangle ^n  |v (s, x)|)^{- \alpha -1} \le (4K) ^{\alpha +1} (1- bs)^{- (\alpha +1) \sigma _1} \\
\langle x \rangle ^n  | L(s,x) |\le \langle x \rangle ^n  | \Delta v (s, x)|\le 4K (1- bs)^{- \sigma _2}
\end{gather*} 
by~\eqref{fWT1}. Since $ (\alpha+1) \sigma
_1 +\sigma_2 \le (\alpha+2) \sigma_{2} \le \sigma_3<1$ by~\eqref{fCA9}, we deduce that
\begin{equation} \label{n110:b2}
\begin{split} 
 | \DIb (x) | ^\alpha  |v (s, x) |^{-\alpha -1}  | L(s, x) | & \le K (4K)^{\alpha +2} (1 - bs)^{- (\alpha +1) \sigma _1 - \sigma _2}  \\ & \le  K (4K)^{\alpha +2} (1 - bs)^{- \sigma _3}.
\end{split} 
\end{equation} 
Thus we see that the integral in~\eqref{n110:b1} is convergent in $L^\infty (\R^N ) $ as $t\uparrow \frac {1} {b}$. It follows that $f$ can be extended to a continuous function $[0, \frac {1} {b}] \to L^\infty  (\R^N ) $ and we set
\begin{equation} \label{fDFz} 
f_0 = f \Bigl( \frac {1} {b} \Bigr) =   - \alpha  \int _0^{\frac {1} {b}}  | \DIb (x) | ^\alpha  |v (s, x) |^{-\alpha -1} L(s, x) \, ds  .
\end{equation} 
We note that by~\eqref{n110:b1}, \eqref{n110:b2} and~\eqref{fDFz}, 
\begin{equation*}
 \| f  (t)  \| _{ L^\infty  } \le \frac {\alpha K (4K)^{\alpha +2} } {b (1- \sigma _3)} 
\end{equation*} 
and
\begin{equation*}
 \| f (t) - f_0 \| _{ L^\infty  } \le  \frac {\alpha K (4K)^{\alpha +2} } {b (1- \sigma _3)} (1- bt)^{1 - \sigma _3}  
\end{equation*} 
for $0\le t\le \frac {1} {b}$.
In particular, if $b_1\ge b_0$ is sufficiently large and $b\ge b_1$, then
\begin{gather} 
 \| f (t) \| _{ L^\infty  } \le  \frac {1} {2}  \label{n118} \\
  \| f (t) - f_0 \| _{ L^\infty  } \le (1-bt)^{1-\sigma _3 }  \label{n114} 
\end{gather} 
for all $0 \le t \le \frac {1} {b}$. 
Therefore, $1+ f_0>0$ by~\eqref{n118}, and it follows from formula~\eqref{fTCA1} that
\begin{equation} \label{n123:1}
0 \le \psi  \le 1.
\end{equation} 
Moreover,  $1 -f(t) \ge \frac {1} {2}$  so that
\begin{equation} \label{n117}
\left\| \frac{1}{1+f (  t ) + \frac{\alpha
|\Im \lambda | } {b} | \DIb | ^{\alpha} | \log(1-bt)| } \right\| _{L^{\infty}}\le 2
\end{equation}
for all $0\le t < \frac {1} {b}$. 
We set
\begin{equation} \label{n123}
 \widetilde{v}  (t,x) = \left(  \frac{ | \DIb (x) | ^{\alpha}} {1+ f_0 (x)+ \frac {\alpha |\Im \lambda |} {b}  |
\DIb (x)  | ^{\alpha} | \log(1-bt)|  } \right)  ^{\frac {1} {\alpha }} .
\end{equation}
It follows from~\eqref{fWTd}  and~\eqref{n117} that
\begin{equation} \label{n123:2}
  \| \langle \cdot \rangle ^n \widetilde{v}(t, \cdot ) \| _{ L^\infty  }\le 2^{\frac {1} {\alpha }}K.
\end{equation} 
In addition, we deduce from~\eqref{n110}, \eqref{n123},  \eqref{n114} and \eqref{n117} (with $t$ and with $t= \frac {1} {b}$) that 
\begin{equation} \label{n111}
 \| \langle \cdot \rangle ^{n \alpha } ( |v (t, \cdot ) |^\alpha -  \widetilde{v}  (t, \cdot ) ^\alpha ) \| _{ L^\infty  } \le 4  \| \langle \cdot \rangle ^n  \DIb \| _{L^{\infty} }^{\alpha} (1-bt)^{1-\sigma_{3}} \le 4  K^{\alpha} (1-bt)^{1-\sigma_{3}} 
\end{equation}
for $0 \le  t<1/b$.  Next, we introduce the decomposition
\begin{equation} \label{n113}
v (  t,x )  =w (  t,x )  \psi (  t,x ) e^{-i\theta (  t,x )  }
\end{equation}
where $\psi $ and $\theta $ are defined by~\eqref{fTCA1} and~\eqref{fTCA2}. 
Differentiating~\eqref{n113} with respect to $t$, we obtain
\begin{equation} \label{fTCA4} 
i w_t = i  \frac {e^{i \theta }} {\psi } v_t - i w \frac {\psi _t} {\psi } - w \theta _t .
\end{equation} 
Moreover, it follows from~\eqref{fTCA1} and~\eqref{n123} that
\begin{equation*} 
 \frac {\psi _t} {\psi } = -  | \Im \lambda | (1-bt) ^{-1}  \widetilde{v} ^\alpha  =   \Im \lambda  (1-bt) ^{-1}  \widetilde{v} ^\alpha 
\end{equation*}
and from~\eqref{fTCA2} and~\eqref{n123} that
\begin{equation*} 
\theta _t = \Re \lambda (1-bt) ^{-1}  \widetilde{v} ^\alpha .
\end{equation*} 
Thus we see that
\begin{equation}  \label{fTCA5} 
- i w \frac {\psi _t} {\psi } - w \theta _t= - \lambda (1-bt) ^{-1}  \widetilde{v} ^\alpha w = - \frac {e^{i\theta }} {\psi } \lambda (1-bt) ^{-1}  \widetilde{v} ^\alpha  v.
\end{equation} 
Formulas~\eqref{fTCA4}, \eqref{fTCA5}  and~\eqref{NLS2} yield
\begin{equation}  \label{fTCA6} 
\begin{split} 
i w_t & =    \frac {e^{i \theta }} {\psi } ( i  v_t - \lambda (1-bt) ^{-1}  \widetilde{v} ^\alpha  v ) \\
& = \frac {e^{i \theta }} {\psi } (- \Delta v + \lambda (1-bt)^{-1} ( |v|^\alpha -  \widetilde{v}  ^\alpha ) v ) .
\end{split} 
\end{equation} 
It follows that
\begin{equation} \label{fTCA7} 
\begin{split} 
 \| \langle \cdot \rangle ^n w_t \| _{ L^\infty  } \le & \|\psi ^{-1} \| _{ L^\infty  }  \| \langle \cdot \rangle ^n \Delta v \| _{ L^\infty  } \\ & +  \|\psi ^{-1} \| _{ L^\infty  }   |\lambda | (1-bt)^{-1} \| \,  |v|^\alpha -  \widetilde{v}  ^\alpha \|  _{ L^\infty  }  \| \langle \cdot \rangle ^n v \| _{ L^\infty  } .
\end{split} 
\end{equation} 
Note that by~\eqref{fTCA1}
\begin{equation*} 
\frac {1} {\psi } =  \Bigl( 1 + \frac { \frac{\alpha |\Im \lambda | } {b} | \DIb | ^{\alpha} | \log(1-bt)| } {1 + f_0} \Bigr)^{ \frac {1} {\alpha }} .
\end{equation*} 
Since $ \| f_0 \| _{ L^\infty  }\le \frac {1} {2}$ by~\eqref{n118}, we deduce that
\begin{equation} \label{fTBB2} 
\|\psi ^{-1} \| _{ L^\infty  }  \le  \Bigl( 1 + 2  \frac{\alpha |\Im \lambda | } {b} \| \DIb \| _{ L^\infty  } ^{\alpha} | \log(1-bt)|  \Bigr)^{ \frac {1} {\alpha }} .
\end{equation} 
Moreover, $\| \langle \cdot \rangle ^n \Delta v \| _{ L^\infty  } \le 4K (1-bt )^ {-\sigma_2 }$ and $  \| \langle \cdot \rangle ^n v \| _{ L^\infty  }  \le 4K$ by~\eqref{fWT1}. Therefore, it follows from~\eqref{fTCA7}, \eqref{fTBB2}  and~\eqref{n111} that
\begin{equation*}
 \| \langle \cdot \rangle ^n w_t \| _{ L^\infty  }  \le C (1+  | \log (1-bt) |)^{\frac {1} {\alpha }} [ (1-bt )^ {-\sigma_2 }  +  (1-bt)^{-\sigma _3} ]  \le C (1-bt) ^{-\sigma _J}
\end{equation*} 
since $\sigma _2< \sigma _3< \sigma _J$. We deduce that
\begin{equation*} 
 \| \langle \cdot \rangle ^n (w (t) - w(s) ) \| _{ L^\infty  } \le C (1- bt) ^{1- \sigma _J}
\end{equation*} 
for all $0\le s<t< \frac {1} {b}$, so that there exists $w_0$ such that $ \langle \cdot \rangle ^n w_0\in L^\infty  (\R^N )  $ and
\begin{equation} \label{n121}
 \| \langle \cdot \rangle ^n (w (t) - w_0 ) \| _{ L^\infty  } \le C (1- bt) ^{1- \sigma _J}
\end{equation} 
for all $0\le t<\frac {1} {b}$.
It follows from~\eqref{n113},  \eqref{n123:1}, and~\eqref{n121} that
\begin{equation} \label{fTTA1} 
 \| \langle \cdot \rangle ^n (v (t, \cdot )  - w_0 (\cdot )  \psi (t, \cdot )  e^{-i \theta (t, \cdot ) } ) \| _{ L^\infty  }     \le C (1- bt) ^{1- \sigma _J}
\end{equation} 
which yields~\eqref{n63}. 
We next prove that $w_0 \not = 0$. 
(Note that  if $\Im \lambda =0$, this is obvious by conservation of the $L^2$ norm.)
Assuming by contradiction that $w_0  = 0$, we deduce from~\eqref{fTTA1} and the property $n>\frac {N} {2}$ that  
\begin{equation} \label{fTBB1} 
\| v(t) \| _{ L^2 } +  \| v(t) \| _{ L^\infty  } \le C (1- bt) ^{1- \sigma _J}. 
\end{equation} 
On the other hand, it follows from equation~\eqref{NLS2} that
\begin{equation*} 
\frac {1} {2} \frac {d} {dt}  \| v(t)\| _{ L^2 }^2 = -\frac {  |\Im \lambda |} { 1-bt } \int  _{ \R^N  }  | v|^{\alpha +2} \ge  -\frac { c } {( 1-bt) ^{1 -  \alpha (1- \sigma _J)} }  \| v(t) \| _{ L^2 }^2 
\end{equation*} 
for some $c>0$, by using the $L^\infty $ estimate of~\eqref{fTBB1}. Therefore,
\begin{equation*} 
 \| v(t)\| _{ L^2 }^2 \ge  \| \DIb \| _{ L^2 }^2 \exp  \Bigl( -2c \int _0^{\frac {1} {b}} \frac {ds} {( 1-bs) ^{1 -  \alpha (1- \sigma _J)}} \Bigr) >0.
\end{equation*} 
This is absurd, since $ \|v(t)\| _{ L^2 } \to 0$ as $t\uparrow \frac {1} {b}$ by the $L^2 $ estimate of~\eqref{fTBB1}. 

We now prove~\eqref{fSPPz}, so we assume $\Im \lambda =0$. The first identity is an immediate consequence of~\eqref{fTCA1}.  Moreover, it follows from~\eqref{fTCA2} that
\begin{equation} \label{fSPPu} 
\theta (  t,x )  = \frac {\lambda } {b}   | \log (1-bt )|  \frac{   | \DIb (x) | ^{\alpha}  } {  1+ f_0 (x)  }    .
\end{equation}
On the other hand, we deduce from~\eqref{n123}  that
$ \widetilde{v}  (t,x) =    (1+ f_0 (x))^{ - \frac {1} {\alpha }}   | \DIb (x)|$, so that~\eqref{n111}  yields
\begin{equation*}
 |v (t, \cdot ) |^\alpha \goto  _{ t\uparrow \frac {1} {b} }  \frac { | \DIb (\cdot )|^\alpha } { 1+ f_0 (\cdot ) } 
 \end{equation*}
in $L^\infty  (\R^N ) $. Since $ |v (  t,x )  | =  |w (  t,x ) |$ by~\eqref{n113} and the first identity in~\eqref{fSPPz}, and $  |w (t, x) | \to  |w_0 (x)|$, we conclude that 
\begin{equation*} 
|w_0 ( \cdot ) |^\alpha = \frac { | \DIb (\cdot )|^\alpha } { 1+ f_0 (\cdot ) } .
\end{equation*} 
The second identity in~\eqref{fSPPz} now follows from~\eqref{fSPPu}. 

If $\Im \lambda =0$, then~\eqref{fp4:1}  is an immediate consequence of~\eqref{n63} and~\eqref{fSPPz}. Assuming now $\Im \lambda <0$, we deduce from~\eqref{n123} that
\begin{equation} \label{fTZ1} 
  | \log(1-bt)| \, \widetilde{v} ^\alpha   =     \frac{ | \DIb   | ^{\alpha}  | \log(1-bt)| } {1+ f_0  + \frac {\alpha |\Im \lambda |} {b}  | \DIb  | ^{\alpha} | \log(1-bt)|  }  .
\end{equation}
Since $1+ f_0\ge 0$ by~\eqref{n118},  it follows in particular that
\begin{equation} \label{fTZ2} 
 | \log(1-bt)| \, \| \widetilde{v} ^\alpha   \| _{ L^\infty  } \le \frac {b} {\alpha  | \Im \lambda |} .
\end{equation} 
Moreover, since $1 + f_0 \le 2$, we deduce from~\eqref{fTZ1} that
\begin{equation*} 
  | \log(1-bt)| \, \widetilde{v} ^\alpha (t, 0)  \ge      \frac{ | \DIb (0)   | ^{\alpha}  | \log(1-bt)| } {2 + \frac {\alpha |\Im \lambda |} {b}  | \DIb (0)  | ^{\alpha} | \log(1-bt)|  }  .
\end{equation*}
Since $ |\DIb (0) | >0$ by~\eqref{fWTd}, it follows that 
\begin{equation} \label{fTZ3} 
 \liminf  _{ t\uparrow \frac {1} {b} } | \log(1-bt)| \, \widetilde{v} ^\alpha (t, 0)  \ge   \frac {b} {\alpha  | \Im \lambda |}  .
\end{equation}
Inequalities~\eqref{fTZ2} and~\eqref{fTZ3} yield
\begin{equation*} 
 | \log(1-bt)| \, \| \widetilde{v} ^\alpha (t, \cdot ) \| _{ L^\infty  } \goto _{ t\uparrow \frac {1} {b} }   \frac {b} {\alpha  | \Im \lambda |}  
\end{equation*}
and~\eqref{fp4:2} follows by applying~\eqref{n111}. 
This completes the proof. 
\end{proof}

\section{Proof of Theorems~$\ref{T1}$ and~$\ref{T2}$} \label{sFIN} 

Let $\DIb \in \Spa$ satisfy~\eqref{fTPA11}, let $K>0$ be sufficiently large so that~\eqref{fWTd} holds, and let $b_1$ be given by Proposition~\ref{p4}.
Given $b\ge b_1$,  let $v\in C([0, \frac {1} {b}), \Spa)$ be the corresponding solution of~\eqref{NLS3}  given by Proposition~\ref{p3}. It is easy to verify that $u$ given by the pseudo-conformal
transformation~\eqref{fNLS1:0} satisfies $u\in
C([0,\infty),{\Sigma}) \cap L^\infty ((0,\infty)\times\R^N )$, and is a solution of~\eqref{NLSI} with $ \DI (x) =e^{i\frac{b|x|^2 }{4}}\DIb (x)$.
Moreover, it follows easily from~\eqref{fWT1} and formula~\eqref{fNLS1:0} that $u\in
 L^\infty ((0,\infty), H^1 (\R^N )  )$. (Here we use the property $n -1>\frac {N} {2}$.)
We now apply Proposition~\ref{p4} and, since $n>\frac {N} {2}$, we deduce from~\eqref{n63} that
\begin{equation} \label{fTPA31} 
\| v (  t , \cdot )  - w_0 (\cdot ) \psi (t, \cdot ) e^{- i \theta (t, \cdot )} \|
_{L^\infty \cap L^2} \le  C(1-bt)^{  1-\sigma_J  }.
\end{equation}

If  $\Im\lambda=0$, then~\eqref{T1n1} follows from~\eqref{fp4:1} and~\eqref{fNLS1:0}; and~\eqref{T1n2:2} follows from~\eqref{fTPA31},  \eqref{fSPPz}, and formula~\eqref{fNLS1:0}. This  proves Theorem~\ref{T1}.

It $\Im\lambda<0$, then~\eqref{T1n3} follows from~\eqref{fp4:2} and~\eqref{fNLS1:0}.
Moreover, it follows from~\eqref{fTCA2}  and~\eqref{fTCA1} that 
\begin{equation} \label{fTCA32} 
\theta (  t,x )  =  \frac {\Re \lambda} {\alpha  | \Im \lambda | } \log
 ( \psi (t,x)^{-\alpha })   =  \frac {\Re \lambda} { \Im \lambda  } \log  ( \psi (t,x) ) .
\end{equation}
Estimate~\eqref{T1n4:2} follows from~\eqref{fTPA31}, \eqref{fTCA32}, and formula~\eqref{fNLS1:0}. This  proves Theorem~\ref{T2}.

\end{document}